\documentclass[11pt, reqno]{amsart}
\usepackage{amsfonts}

\usepackage{amscd,amssymb,amsmath,graphicx,verbatim,mathrsfs}
\usepackage[colorlinks=true,linkcolor=blue,citecolor=blue]{hyperref}
\usepackage{extarrows}
\usepackage[all]{xy}
\usepackage{dcpic,mathtools,pictexwd}
\usepackage{rotating}

\usepackage[left=1.4in,right=1.4in,top=1.4in,bottom=1.4in]{geometry}

\theoremstyle{remark}

\usepackage{tikz-cd}
\usepackage{tikz}
\usetikzlibrary{arrows.meta}
\usepackage{adjustbox}

\usepackage{setspace}

\theoremstyle{plain}

\newtheorem{thm}{Theorem}[section]
\newtheorem{lem}[thm]{Lemma}

\newtheorem{prop}[thm]{Proposition}
\newtheorem{conj}[thm]{Conjecture}

\theoremstyle{definition}

\newtheorem{defn}[thm]{Definition}

\newtheorem{rmk}[thm]{Remark}

\numberwithin{equation}{section}

\newcommand{\pg}{\mathrm{pg}}
\newcommand{\id}{\mathrm{id}}
\newcommand{\red}{r}
\newcommand{\maximal}{\mathrm{max}}
\newcommand{\asymalg}{\mathfrak{A}}
\newcommand{\ev}{\mathrm{ev}}

\makeatletter
\def\Eqlfill@{\arrowfill@\Relbar\Relbar\Relbar}
\newcommand{\extendEql}[1][]{\ext@arrow 0359\Eqlfill@{#1}}
\makeatother

\begin{document}
	
	\title{Strong relative Novikov conjecture for coarsely embeddable groups}

	\author{Geng Tian}
	\address[Geng Tian]{School of Mathematics and statistics, Liaoning University}
	\email{gengtian.ncg@gmail.com}

	\author{Zhizhang Xie}
	\address[Zhizhang Xie]{ Department of Mathematics, Texas A\&M University }
	\email{xie@math.tamu.edu}
	\thanks{The second author is partially supported by NSF  1952693.}
	
	\author{Guoliang Yu}
	\address[Guoliang Yu]{ Department of
		Mathematics, Texas A\&M University}
	\email{guoliangyu@math.tamu.edu}
	\thanks{The third author is partially supported by NSF 2000082.}
	
	\date{}	\maketitle

	%\renewcommand{\thefootnote}{\fnsymbol{footnote}}
	%\footnotetext[1]{Corresponding author.}

	\date{}
	
	\thanks{}
	\begin{abstract}
		In this article, we prove a strong relative Novikov conjecture for any pair of groups that are coarsely embeddable into Hilbert space.
	\end{abstract}

	\maketitle
	
	%\tableofcontents

	\section{Introduction}

	A fundamental problem in  topology is the Novikov conjecture which states that the higher signatures of a closed (i.e. compact without boundary) oriented smooth manifold are invariant under orientation-preserving homotopy equivalences. The Novikov conjecture has been proved for a large class of manifolds by techniques from noncommutative geometry and geometric group theory \cite{Yu_coa_emb}. While the Novikov conjecture concerns the homotopy invariance of higher signatures of closed manifolds, it has  a natural analogue for compact oriented manifolds with boundary. This is usually  called
	the relative Novikov Conjecture, which states that  the relative higher signatures of a compact oriented manifold with boundary are invariant under orientation-preserving homotopy equivalences of manifolds with boundary. Here  a homotopy equivalence $f\colon (M, 
	\partial M) \to (N, \partial N)$ between two manifolds with boundary means $f$, its homotopy inverse and the relevant homotopies  map $\partial M$ to $\partial N$. The main purpose of this article is to develop a $C^*$-algebraic approach to the relative Novikov conjecture. More precisely,  we shall prove a $C^\ast$-algebraic version of the relative Novikov conjecture (which is usually called the strong relative Novikov conjecture) for any pair of groups that are coarsely embeddable into Hilbert space (Theorem \ref{main-theorem}).
	
%In particular, we prove that the relative Novikov conjecture holds for any compact oriented smooth manifold with boundary, provided that  the fundamental groups of the manifold and its boundary  are coarsely embeddable into Hilbert space. More precisely, we prove the following theorem. 
%	
%	\begin{thm}\label{thm:relativeNovikov}
%		Let $(M, \partial M)$  and $(N, \partial N)$ be compact oriented smooth manifolds with boundary. Suppose   $f\colon (M, \partial M)\to (N, \partial N)$ is an  orientation-preserving homotopy equivalence of manifold pairs.  Denote $G=\pi_1 (\partial M)\cong \pi_1 (\partial N)$ and $\Gamma=\pi_1M\cong \pi_1N$. If both $G$ and $\Gamma$ are coarsely  embeddable into Hilbert space, then the relative Novikov conjecture holds for $(G, \Gamma)$, i.e.,  the relative higher signatures of $(M, \partial M)$ and $(N, \partial N)$ are invariant under the homotopy equivalence $f$.
%	\end{thm}

	Before we explain our main theorem (Theorem \ref{main-theorem}) and the main difficulties of its proof, let us first fix some notation. Suppose $G$ is a countable discrete group. Let  $\underline{E}G$  be  the universal space for proper $G$  actions and  $K_{*}^{G}(\underline{E}G)$ be the $G$-equivariant $K$-homology of $\underline{E}G$. 
	The  Baum-Connes assembly map for  $G$ is denoted by 
	\[
	\mu\colon  K^G_*(\underline{E}G) \rightarrow K_*(C^*_{\red}(G))
	\]
	where $C^*_{\red}(G)$ is the reduced group $C^\ast$ algebra of $G$. A theorem of the third author showed that the Baum-Connes assembly map $\mu$ is injective when $G$ is coarsely embeddable into Hilbert space \cite{Yu_coa_emb}. 
	
	Now suppose  $h\colon G\to \Gamma$ is a group homomorphism between two countable discrete groups.   The map  $h\colon G \to \Gamma$ naturally induces a map from $\underline{E}G$ to $\underline{E}\Gamma$, which will still be denoted by $h$.  We denote by $K_{*}^{G,\Gamma}(\underline{E}G,\underline{E}\Gamma)$ the relative $K$-homology of $(\underline{E}G,\underline{E}\Gamma)$ with respect to the map $h\colon \underline{E}G  \to \underline{E}\Gamma$. We certainly would like to have the analogue of the Baum-Connes assembly map for the relative $K$-homology. The first obstacle for having such a relative Baum-Connes assembly map is that a group homomorphism $h\colon G\to\Gamma$ does \emph{not} induce a $C^*$-homomorphism between the reduced group $C^*$ algebras $C_{\red}^{*}(G)$ and $C_{\red}^{*}(\Gamma)$. Of course, this can be easily resolved by using the  maximal group $C^\ast$-algebras instead, which gives the following maximal relative Baum-Connes assembly map: 
	\[ \mu_{\maximal} \colon K_{*}^{G,\Gamma}(\underline{E}G,\underline{E}\Gamma) \rightarrow K_*(C_{\maximal}^{*}(G,\Gamma)) \]
	where  $C_{\maximal}^{*}(G,\Gamma)$ is the mapping cone of $h\colon  C_{\maximal}^{*}(G) \to C_{\maximal}^{*}(\Gamma)$. The map $\mu_{\maximal}$ fits into the following  commutative diagram of  long exact sequences:  
	\[
	\adjustbox{scale=0.9,center}{  \begin{tikzcd}
			K^G_{i+1}(\underline{E}G) & K^\Gamma_{i+1}(\underline{E}\Gamma) & K_{i}^{G,\Gamma}(\underline{E}G,\underline{E}\Gamma)  & K^G_{i}(\underline{E}G) & {} \\
			K_{i+1}(C^*_{\maximal}(G)) & K_{i+1}(C^*_{\maximal}(\Gamma)) & K_i(C_{\maximal}^{*}(G,\Gamma))  & K_i(C^*_{\maximal}(G)) & {}
			\arrow[from=1-1, to=1-2]
			\arrow[from=1-2, to=1-3]
			\arrow[from=2-1, to=2-2]
			\arrow[from=2-2, to=2-3]
			\arrow["\mu^G_{\maximal}", from=1-1, to=2-1]
			\arrow["\mu^\Gamma_{\maximal}", from=1-2, to=2-2]
			\arrow["\mu_{\maximal}", from=1-3, to=2-3]
			\arrow[from=1-3, to=1-4]
			\arrow[from=2-3, to=2-4]
			\arrow["\mu^G_{\maximal}", from=1-4, to=2-4]
			\arrow[from=1-4, to=1-5]
			\arrow[from=2-4, to=2-5]
		\end{tikzcd}
	}\]
	In order to see when the map $\mu_{\maximal}$ in the third column is injective, one natural attempt would be to apply the five lemma.  However, this would require $\mu^G_{\maximal}$ in the first column to be surjective,  and $\mu^\Gamma_{\maximal}$ in the second column and  $\mu^G_{\maximal}$ in the fourth column to be injective. Although the injectivity of  $\mu^G_{\maximal}$ is known for all groups that are coarsely embeddable into Hilbert space, the surjectivity of $\mu^G_{\maximal}$  in fact fails in general. For example,  $\mu^G_{\maximal}$ fails to be surjective when $G$ has property (T). 
	
	A key new ingredient of the present paper is the construction of a new relative $C^\ast$-algebra  that circumvents the above obstacle. The precise definition of this new relative $C^\ast$-algebra requires the usage of asymptotic morphisms,\footnote{See Section \ref{sec:asymptotic} for a brief review of the basics of asymptotic morphisms.} and will be given in   Lemma \ref{key-lemma} and Definiton \ref{relative-gamma-reduced-algebras}. For the moment, let us summarize the key properties of the new relative $C^*$ algebra that are central to the proofs of the main theorems in the present paper  and introduce the following more conceptual notition of admissible asymptotic morphisms. 
    \begin{defn}\label{def:asym-admissible}
        Let $h\colon G\to \Gamma$ be a group homomorphism between two countable discrete groups. Let $\Sigma^2 C_r^*(G) = C_r^*(G)\otimes C_0(\mathbb R^2)$ is the double suspension of $C_r^*(G)$ and $\mathcal K$ be the algebra of compact operators on a Hilbert space.  An asymptotic morphism\ $\phi\colon \Sigma^2 C_r^*(G) \dashrightarrow C_r^*(\Gamma)\otimes \mathcal K$ is said to be \emph{admissible} if there exists an asymptotic morhpism $\phi_L\colon \Sigma^2 C_L^*(\underline{E}G)^G  \dashrightarrow C_L^*(\underline E\Gamma)^\Gamma\otimes \mathcal K$ such that  
        \[ (\phi_L)_\ast = h_\ast \colon K_i(\Sigma^2 C_L^*(\underline{E}G)^G) = K_i^G(\underline{E}G) \to K_i(C_L^*(\underline E\Gamma)^\Gamma\otimes \mathcal K) = K_i^\Gamma(\underline E\Gamma)\]
        and the following diagram commutes: 
       \[ \begin{tikzcd}
			\Sigma^2 C_L^*(\underline{E}G)^G  & C_L^*(\underline E\Gamma)^\Gamma\otimes \mathcal K \\
			\Sigma^2 C_r^*(G) & C_r^*(\Gamma)\otimes \mathcal K
			\arrow["\phi_L", dashed, from=1-1, to=1-2]
			\arrow["\ev", from=1-1, to=2-1]
			\arrow["\ev", from=1-2, to=2-2]
			\arrow["\phi", dashed, from=2-1, to=2-2]
		\end{tikzcd}
       \]
 Here $C_L^*(\underline{E}G)^G$ is the $G$-equivariant localiztion algebra associated to $\underline EG$  and $\ev$ is the evaluation map at $t=0$ (see for example Section \ref{sec:Roe-local}).
    \end{defn}

Our main  examples of admissible asymptotic morphisms come from $\gamma$ elements (in the sense of Kasparov).   In particular, if $G$ has a $\gamma$ element, we shall  construct a relative $C^*$ algebra   $C_{\gamma}^{*}(G,\Gamma)$ and  also make sense of the corresponding relative  Baum-Connes assembly map for $C_{\gamma}^{*}(G,\Gamma)$, and finally  prove the following version of $C^\ast$-algebraic relative Novikov conjecture. 
	
	\begin{thm}\label{main-theorem}
		Let  $h\colon G\rightarrow \Gamma$ be a group homomorphism between two countable discrete groups.  If both $G$ and $\Gamma$ are coarsely embeddable  into Hilbert space,
		then the  strong relative Novikov conjecture holds for $(G, \Gamma, h)$, that is, the relative Baum--Connes assembly map
		$$
		\mu\colon K_{*}^{G,\Gamma}(\underline{E}G,\underline{E}\Gamma) \rightarrow K_*(C_{\gamma}^{*}(G,\Gamma)) 
		$$
		is  injective.
	\end{thm}

	The paper is organized as follows. In Section \ref{sec:pre}, we review some basics of $KK$-theory, $E$-theory, and their connections.  In section \ref{sec:gamma-groupalgebra}, we introduce  the notion of relative $\gamma$-reduced $C^*$-algebras. In  Section \ref{sec:relativetheory}, we introduce the corresponding version of relative $K$-homology, and verify it coincides with the standard relative $K$-homology. We then introduce the corresponding relative Baum-Connes assembly map for the relative $\gamma$-reduced $C^*$-algebras.    In section \ref{sec:relativeBott}, we construct a relative Bott map in the context of relative $\gamma$-reduced $C^*$-algebras. In Section \ref{sec:main}, we prove the main theorems of the paper.

	\section{Preliminaries}\label{sec:pre}
	In this section, we review some basic constructions that will be needed in the later part of the paper.

	\subsection{Roe algebras and localization algebras}\label{sec:Roe-local}
	In this subsection, we recall the definitions of Roe algebras and localization algebras for a metric space $Z$ endowed with a proper $G$-action (cf. \cite{Willett-Yu-higher-index-book}).
	
	Let $Z$ be a metric space with a proper $G$-action by isometries, and $A$ a $G$-$C^*$-algebra. A $G$-action on $Z$ is said to be proper if for every $z \in Z$,
	\begin{center}
		$d(z,gz)\to \infty$, as $g \to \infty$.
	\end{center}
	A $G$-action is said to be cocompact if the quotient space $Z/G$ is compact.
	Let $H$ be a $G$-Hilbert module over $A$, and $\varphi: C_0(Z)\to B(H)$ a $*$-representation, where $B(H)$ is the $C^*$-algebra of all bounded (adjointable) operators on $H$. The triple $(C_0(Z), G, H)$ is called a covariant system if
	$$
	\varphi(\gamma f)v=(\gamma \varphi(f)\gamma^{-1})v
	$$
	for all $\gamma \in G$, $f \in C_0(Z)$ and $v \in H$.
	
	\begin{defn} Let $H$ be a $G$-Hilbert module over $A$ and $(C_0(Z),G, H)$ a covariant system. Let $T: H \to H$ be an adjointable operator.
		\begin{enumerate}
			\item[(1)] The support of $T$, denoted by $\mbox{Sup}(T)$, is defined to be the complement of the set of all points $(x, y)\in Z \times Z$ for which there exists $f \in C_0(Z)$ and $g \in C_0(Z)$ satisfying $f\cdot T\cdot g\neq 0$, and $f(x)\neq 0$ and $g(y)\neq 0$;
			\item[(2)] The propagation of the operator $T$ is defined by
			$$\pg(T)=\sup \left\{ d(x,y): (x,y) \in \mbox{Supp}(T)\right\}.$$
			An operator $T$ is said to have finite propagation if $\pg(T)< \infty$;
			\item[(3)] The operator $T$ is locally compact if $f\cdot T$ and $T \cdot f$ are compact operators over the Hilbert module $H$ for all $f \in C_0(X)$, where an operator is said to be compact if it is an approximation of finite rank operators.
			\item[(4)] The operator $T$ is $G$-invariant if $g\cdot T=T\cdot g$ for all $g \in G$.
		\end{enumerate}
	\end{defn}

	\begin{defn}
		An admissible Hilbert $G$-$Z$-module is a covariant system $(C_0(Z), G, H)$ satisfying
		\begin{enumerate}
			\item[(1)] the $G$-action on $Z$ is proper and cocompact;
			\item[(2)] there exists a $G$-Hilbert space  $H_Z$ satisfying the following: 
			\begin{itemize}
				\item $H$ is isomorphic to $H_Z \otimes A$ as $G$-Hilbert modules over $A$;
				\item $\varphi=\varphi_0 \otimes I$ where  $\varphi_0\colon C_0(Z) \to B(H_Z)$ is a $G$-equivariant $*$-homomorphism  such that $\varphi_0(f)$ is not in $K(H_Z)$ for any non-zero function $f \in C_0(Z)$ and $\varphi_0$ is non-degenerate in the sense that $\left\{\varphi_0(f)v: v \in H_X, f \in C_0(Z)\right\}$ is dense in $H_Z$;
				\item for any finite subgroup $F \subseteq  G$ and any $F$-invariant Borel subset $E$ of $Z$, there is Hilbert space $H_E$ with trivial $F$-action such that $\chi_E H_Z$ and $\ell^2(F) \otimes H_E$ are isomorphic as representations of $F$.
			\end{itemize}
			
		\end{enumerate}
	\end{defn}

	\begin{defn}
		Let $(C_0(Z), G, H)$ be an admissible system. The algebraic Roe algebra with coefficients in $A$, denoted by $\mathbb{C}[Z, A]^G$, is defined to be the $C^*$-subalgebra of $B(H)$ consisting of $G$-invariant, locally compact operators with finite propagation. The Roe algebra $C^*_{\red}(Z, A)^G$ is completion of the $*$-algebra $\mathbb{C}[Z, A]^G$ under the operator norm in $B(H)$.
	\end{defn}
	
	It is easy to show that the definition of algebraic Roe algebra is independent on the choice of admissible systems  (cf. \cite{Willett-Yu-higher-index-book}).
	
	The following result follows from the similar arguments in  \cite[Lemma 3.4]{Gong-Wang-Yu}.
	\begin{lem}[{\cite[Lemma 3.4]{Gong-Wang-Yu}}]
		Let $(C_0(Z), G, H)$ be an admissible system. For each $T \in \mathbb{C}[Z, A]^G$, there exist a constant $C>0$ such that
		$$\|\pi(T)\|\leq C$$
		for all $*$-representations $\pi\colon \mathbb{C}[Z, A]^G \to B(H)$.
	\end{lem}
	
	It follows from the above result that the maximal norm on the $*$-algebra $\mathbb{C}[Z, A]^G$ is well-defined. We then defined the maximal Roe algebra, denoted by $C^*_{max}(Z, A)^G$, to be the completion of $\mathbb{C}[G, Z, A]$ under the maximal norm
	$$\|T\|_{\maximal}=\sup \left\{\|\pi(T)\| 
	\mid \pi\colon \mathbb{C}[Z, A]^G \to B(H) \mbox{~is~a~} *\mbox{-representation}\right\}.$$
	
	\begin{defn}\leavevmode
		\begin{enumerate}
			\item[(1)] The algebraic maximal algebraic localization algebra $\mathbb{C}_{max,L}[Z, A]^G$ is defined to be the $*$-algebra of all uniformly bounded and uniformly continuous functions $f:[0, \infty)\to C^*_{\maximal}(Z, A)^G$ such that
			$$\pg(f(t))\to 0, \text{~~as~~} t \to \infty.$$
			\item[(2)] The maximal localization algebra $C_{\maximal,L}^*(Z, A)^G$ is defined to be the completion of $\mathbb{C}_{\maximal,L}[Z, A]^G$ under the norm
			$$\|f\|=\sup_{t\in [0,\infty)} \|f(t)\|_{\maximal},$$
			for all $f \in \mathbb{C}_{\maximal,L}[Z, A]^G$.
			\item[(3)] The algebraic localization algebra $\mathbb{C}_{L}[Z, A]^G$ is defined to be the uniformly bounded and uniformly continuous functions $f:[0, \infty)\to C^*_{\red}(Z, A)^G$ such that
			$$\pg(f(t))\to 0, \text{~~as~~} t \to \infty.$$
			\item[(4)] The localization algebra $C_L^*(Z, A)^G$ is defined to be the completion of $\mathbb{C}_{L}[Z, A]^G$ under the norm
			$$\|f\|=\sup_{t\in [0,\infty)} \|f(t)\|,$$
			for all $f \in \mathbb{C}_{L}[Z, A]^G$.
		\end{enumerate}
	\end{defn}
	
	By the universality of the maximal norm, the identity map on the algebraic Roe algebra $\mathbb{C}[G, Z,A]$ extends to a $*$-homomorphism
	$$\lambda\colon  C_{\maximal}^*(Z, A)^G \to C^*_{\red}(Z, A)^G.$$
	Similarly, the identity map on the algebraic localization algebras extends to a $*$-homomorphism
	$$\lambda_L\colon  C_{\maximal, L}^*(Z, A)^G \to C_L^*(Z, A)^G.$$
	Note that   the map $\lambda_{L}$ always induces an isomorphism at the level of  $K$-theory.

	We have the evaluation map from the maximal localization algebra to the maximal Roe algebra
	$$
	\ev_{\maximal}\colon C_{\maximal, L}^*(Z, A)^G \to C_{\maximal}^*(Z, A)^G
	$$
	by
	$$\ev_{\maximal}(f)=f(0)$$
	for all $f \in C_{\maximal, L}^*(Z, A)^G$.
	Similarly, we also have the evaluation map from the reduced localization algebra to the reduced Roe algebra
	$$
	\ev \colon C_{L}^*(Z, A)^G \to C^*_{\red}(Z, A)^G.
	$$
	%These evaluation maps induce homomorphisms
	%$$
	%e_{max,*}:K_*(C_{L}^*(Z, A)^G) \to K_*(C_{max}^*(Z, A)^G)
	%$$
	%and
	%$$
	%e_{*}:K_*(C_{L}^*(Z, A)^G) \to K_*(C^*_{\red}(Z, A)^G) $$
	%at the level of $K$-theory. 

	\subsection{Equivariantly uniformly asymptotic morphisms}\label{sec:asymptotic}
	In this subsection, we review the notion of asymptotic morphisms, due to Connes and Higson \cite{MR1065438}. See also \cite{MR1711324}.

	\begin{defn}\label{def:asymptotic}
		Let $G$ be a countable discrete group, and  $A$ and $B$ two $G$-$C^*$-algebras. A $G$-equivariant uniformly asymptotic morphism from $A$ to $B$ is a family of functions $\phi_t:A\rightarrow B$, $ t\in[1,\infty)$ satisfying that
		
		\begin{enumerate}
			\item [(1)] for any $a\in A$, the map $t\mapsto \phi_t(a)\colon [1,\infty)\rightarrow B$  is bounded and  uniformly norm-continuous;
			\item[(2)] for any $a,a_1,a_2\in A$, $g\in G$ and $\lambda\in \mathbb{C}$, 
			$$ \lim\limits_{t\rightarrow \infty}\left\{\begin{array}{cccc}
				\phi_t(a_1a_2)-\phi_t(a_1)\phi_t(a_2)\\
				\phi_t(a_1+a_2)-(\phi_t(a_1)+\phi_t(a_2))\\
				\phi_t(\lambda a)-\lambda\phi_t(a)\\
				\phi_t(a^*)-\phi_t(a)^*\\
				\phi_t(ga)-g\phi_t(a)\\
			\end{array} \right\} = 0.$$
		\end{enumerate}
		We denote an asymptotic morphism above by the notation $$\phi\colon A\dashrightarrow B.$$
	\end{defn}
	
	\begin{rmk}
		Note that we have imposed the uniform continuous condition in the above definition of asymptotic morphisms. This was not assumed in the original  definition of asymptotic morphisms by Connes--Higson. 
	\end{rmk}

	\begin{defn}
		Two asymptotic morphisms $\phi, \varphi\colon A\dashrightarrow B$ are (asymptotically) equivalent if for all $a\in A$, $$\lim\limits_{t\rightarrow \infty}||\phi_t(a)-\varphi_t(a)||=0.$$
	\end{defn}
	
	Up to asymptotic equivalence, a $G$-equivariantly uniformly asymptotic morphism $\phi\colon A\dashrightarrow B$ is the same as a $G$-equivariant $*$-homomorphism from $A$ into the following \emph{uniformly asymptotic $C^*$-algebra} associated to $B$.
	
	\begin{defn}\label{uniformly-asymptotic-algebra}
		Let $B$ be a $G$-$C^*$-algebra. Denote by $C_{buc}([1,\infty),B)$ the $C^*$-algebra of bounded, uniformly continuous functions from $[1,\infty)$ into $B$, and denote by $C_0([1,\infty),B)$  the $C^*$-subalgebra consisting of  functions which vanish at infinity. Note that  $C_0([1,\infty),B)$ is an ideal of $C_{buc}([1,\infty),B)$. We define the uniformly asymptotic $C^*$-algebra of $B$ to be  the quotient $C^*$-algebra $$\asymalg(B):=C_{buc}([1,\infty),B)/C_0([1,\infty),B).$$
	\end{defn}

	If $\varphi:A\rightarrow \asymalg(B)$ is a $G$-equivariant $*$-homomorphism,  then by composing $\varphi$ with a set-theoretic section $\asymalg(B)\to  C_{buc}([1,\infty),B)$ of the quotient map,  we obtain an $G$-equivariant uniformly  asymptotic morphism from
	$A$ to $B$; its equivalence class is independent of the choice of the section. Conversely, a $G$-equivariant uniformly asymptotic morphism $\phi\colon A\dashrightarrow B$ can be viewed as a map from $A$ into $C_{buc}([1,\infty),B)$, and by composing with the quotient map to $ \asymalg(B)$, we obtain a $G$-equivariant $*$-homomorphism from $A$ to $ \asymalg(B)$,  which depends only on the asymptotic equivalence class of $\phi$.

	%\begin{rmk}
	%If there is no  $G$-action in above setting, one can remove ``equivariantly'' everywhere. If the uniformly condition is removed, then it will recover the classical one defined by Connes-Higson.
	%\end{rmk}

	\begin{defn} \label{push} A pull-back diagram of $C^*$-algebras is a diagram of the form
		\begin{equation*}
			\begin{tikzcd}
				P\arrow[r,"p^C"]\arrow{d}[swap]{P^D}&C \arrow[d,"\pi^C"]\\
				D\arrow{r}{}[swap]{\pi^D}& E
			\end{tikzcd}
		\end{equation*}
		such that  $\pi^C$ and $\pi^D$ are surjections, $P=\{(c,d)\in C\oplus D| \pi^C(c)=\pi^D(d)\}$, $p^C$ and $p^D$  are the obvious projections.
	\end{defn}
	
	Each pull-back diagram of $C^*$-algebras induces a corresponding Mayer-Vietoris sequence in $K$-theory.

	\begin{prop}[{cf. \cite[Proposition 2.7.15]{Willett-Yu-higher-index-book}}]\label{pull-back-diagram}
		Given a pull-back diagram of $C^*$-algebras as  in Definition \ref{push}, we have the following six-term exact sequence:
		\[
		\begindc{\commdiag}[100]		% [10]
		\obj(0,5)[2a]{$K_1(P)$}	\obj(10,5)[2b]{$K_1(C)\oplus K_1(D)$}	\obj(20,5)[2c]{$K_1(E)$}
		
		\mor{2a}{2b}{}	\mor{2b}{2c}{}
		
		\obj(0,0)[3a]{$K_0(E)$}	\obj(20,0)[3c]{$K_0(P)$}	\obj(10,0)[3b]{$K_0(C)\oplus K_0(D)$}
		
		\mor{3a}{2a}{}	\mor{2c}{3c}{}		\mor{3b}{3a}{}	\mor{3c}{3b}{}
		\enddc
		\]
		where the morphisms $$K_*(P)\rightarrow K_*(C)\oplus K_*(D) ~\text{ and }~ K_*(C)\oplus K_*(D) \rightarrow K_*(E)$$ are given by $$x \mapsto p^C_*(x)\oplus p^D_*(x)~ \text{ and }~ y\oplus z\mapsto \pi^C_*(y)-\pi^D_*(z)$$ respectively.
	\end{prop}
	
	From now on, for each $f\in C_{buc}([1,\infty),B)$, we shall denote the corresponding element in $\asymalg(B)$ by $[f]$.  Note that there is a natural asymptotic morphism 
	\begin{equation}\label{eq:canonicalaysm}
		\pi\colon  \asymalg(B)\dashrightarrow B
	\end{equation} given  by $$\pi_t([f])=f(t).$$ The following lemma shows that $\pi$ induces an isomorphism on $K$-theory. 
	
	\begin{lem}\label{reversetrick}
		Let $B$ be a $C^*$-algebra, then the obvious inclusion $\theta\colon B \hookrightarrow \asymalg(B)$ induces an isomorphism $$\theta_*\colon K_*(B) \rightarrow K_*(\asymalg(B)).$$  Its inverse map is
		$$\pi_*\colon K_*(\asymalg(B))\rightarrow K_*(B),$$ where $\pi$ is the asymptotic morphism from line \eqref{eq:canonicalaysm}.
	\end{lem}
	\begin{proof}
		Since $K_*(C_0([1,\infty),B))=0$, it follows from the short exact sequence
		$$0\rightarrow C_0([1,\infty),B) \to C_{ubc}([1,\infty),B)\to  C_{ubc}([1,\infty),B)/C_0([1,\infty),B)\to 0$$
		that $K_*(C_{ubc}([1,\infty),B))\to  K_*(\asymalg(B))$ is an isomorphism.  It suffices to prove that
		$$K_*(B)\cong K_*(C_{ubc}([1,\infty),B)).$$

		Now let $$\ev\colon C_{ubc}([1,\infty),B)\rightarrow B$$ be the evaluation map at $t=1$.
		
		$$0\rightarrow \ker(\ev)\rightarrow C_{ubc}([1,\infty),B)\rightarrow B\rightarrow 0$$ It suffices to prove  $K_*(\ker({ev}))=0$.
		
		For brievity, let us denote
		$$C=C_0((1,2],B)\oplus C_{ubc}(\bigsqcup\limits_{n~ odd} [n,n+1], B)$$  $$D=C_{ubc}(\bigsqcup\limits_{n~ even} [n,n+1], B)$$
		and
		$$E=C_{b}(\bigsqcup\limits_{n\in \mathbb{N}, ~n\geqslant 2} \{n\}, B).$$
		We have the following pull-back  diagram
		of $C^\ast$-algebras: 
		\begin{equation*}
			\begin{tikzcd}
				\ker(\ev)\arrow[r]\arrow[d]& \arrow[d]C\\
				D\ar[r]& E.\\
			\end{tikzcd}
		\end{equation*}
		By Proposition \ref{pull-back-diagram}, it induces the following six-term exact sequence
		\begin{equation*}
			\begin{tikzcd}
				K_1(\ker(\ev))\arrow[r]&K_1(C)\oplus K_1(D)\arrow[r]&K_1(E)\arrow[d]\\
				K_0(E)\arrow[u]& K_0(C)\oplus K_0(D)\arrow[l] & K_0(\ker(\ev))\arrow[l]
			\end{tikzcd}
		\end{equation*}
		
		Note that the algebras $C,D$ and $E$ are homotopic equivalent to the following algebras 
		$$\prod\limits_{n~odd, n\geqslant 3}B, \prod\limits_{n~even, n\geqslant 2}B ~ \text{ and }\prod\limits_{n\in \mathbb{N}, ~n\geqslant 2} B,$$
		respectively. Thus, we have the following exact sequence:
		\begin{equation*}
			\adjustbox{scale=0.9,center}{  \begin{tikzcd}
					K_1(\ker(\ev))\arrow[r]& \prod\limits_{n~odd, n\geqslant 3}K_1(B)\oplus  \prod\limits_{n~even, n\geqslant 2} K_1(B)\arrow[d]\arrow[r,"\tau"]&\prod\limits_{n\in \mathbb{N}, ~n\geqslant 2}  K_1(B)\arrow[d]\\
					\prod\limits_{n\in \mathbb{N}, ~n\geqslant 2}  K_0(B)\arrow[u]& \prod\limits_{n~odd, n\geqslant 3}K_0(B)\oplus  \prod\limits_{n~even, n\geqslant 2} K_0(B)\arrow[l,swap,"\tau"]& K_0(\ker(\ev)) \arrow[l].
				\end{tikzcd}
			}
		\end{equation*}
		The morphisms in the above diagram are explicitly described in  Proposition \ref{pull-back-diagram}. In particular, 
		the morphism
		$$\tau\colon \prod\limits_{n~odd, n\geqslant 3}K_*(B) \oplus \prod\limits_{n~even, n\geqslant 2} K_*(B) \longrightarrow \prod\limits_{n\in \mathbb{N}, ~n\geqslant 2}  K_*(B)$$
		maps $(a_3,a_5,\cdots)\oplus (a_2,a_4,\cdots)$ to  $(-a_2,a_3-a_2, a_3-a_4, a_5-a_4,\cdots)$.
		It is easy to check that $\tau$  is an isomorphism.  Therefore, $K_*(\ker(ev))=0$. 
		To summarize, we have shown that $\theta_*$ is isomorphism. 
		
		Clearly, it follows by construction  that $\pi_*\circ\theta_*=\id.$  Since we have shown $\theta_*$ is isomorphism, it follows that  $\pi_*$ is the inverse of $\theta_\ast$. 
	\end{proof}

	\begin{rmk}
		Notice that the composition of a uniformly asymptotic morphism and a genuine homomorphism is still a uniformly asymptotic morphism. However,  in order to make the composition of  two asymptotic morphisms into an asymptotic morphism, usually a reparametrization (which always exists) is needed. More precisely, given two uniformly asymptotic morphisms, $\alpha_t$ and $\beta_t$, there exists a (continuous and increasing) reparametrization $s(t)$ of $t$, such that $\alpha_{s(t)}\circ\beta_t$ is a uniformly asymptotic morphism.   
	\end{rmk}
	
	\subsection{$KK$-theory  and its relation with $E$-theory}
	
	In this subsection, we review some basics of  Kasparov's equivariant $KK$-theory (cf. \cite{Kasparov}) and its relation with $E$-theory. $KK$-theory associates an abelian group, denoted by  $KK^G(A, B)$, to each pair of two separable $G$-$C^*$-algebras $A$ and $B$.  It is contravariant in $A$ and covariant in $B$. It is $G$-equivariant-homotopy-invariant, stably invariant, preserves $G$-equivariant split exact sequences, and satisfies \emph{Bott periodicity}. Here Bott periodicity means that  there are natural isomorphisms
	\[
	KK^G(A, B) \cong KK^G(\Sigma  ^2 A, B) \cong KK^G(\Sigma   A, \Sigma   B)  \cong KK^G(A, \Sigma  ^2 B)
	\]
	where $\Sigma^k A$ stands for $C_0(\mathbb{R}^k,A)$ with $G$ acting trivially on $\mathbb{R}^k$, for each $k\in \mathbb{N}$. For each short exact sequence
	$$0 \rightarrow J \rightarrow A \rightarrow A/J\rightarrow 0$$
	of $G$-$C^*$-algebras,  there is a natural \emph{six-term exact sequence}. Equivariant $KK$-theory is a far-reaching generalization of both $K$-theory and $K$-homology. In particular, if either $A$ or $B$ is $\mathbb{C}$, we have 
	\begin{itemize}
		\item equivariant $K$-theory: $KK^G(\mathbb{C}, B) \cong K^G_0(B)$;
		\item equivariant $K$-homology: $KK^G(A, \mathbb{C}) \cong K_G^0(A)$.
	\end{itemize}

	When the acting group $G$ is a trivial group or the action is trivial, we simply write $KK(A,B)$ for $KK^G(A, B)$ and drop the word ``equivariant'' everywhere. There is a forgetful functor from $KK^G$ to $KK$.

	\begin{rmk} \label{rmk:KK-facts-de-equivariantize}
		In some important special cases, we can turn an equivariant $KK$-group $KK^G(A, B)$ into a related non-equivariant $KK$-group, the latter of which sometimes is easier to study.
		\begin{enumerate}
			\item\label{rmk:KK-facts-de-equivariantize:trivial-B} When $G$ is a countable discrete group and its action on $B$ is trivial, it is immediate from the definition that there is a natural isomorphism $$KK^G(A, B) \cong KK(C^*_{max}(G,A), B)$$
			where $C^*_{max}(G,A)$ is the maximal crossed product. In particular, if $A = C_0(X)$ for a locally compact, second countable Hausdorff space $X$ and $G$ acts freely and properly on $X$, then since $C^*_{max}(G,C_0(X))$ is stably isomorphic to $C_0(X / G)$, we have a natural isomorphism
			$$KK^G(C_0(X), B) \cong KK(C_0(X / G), B).$$
			\item\label{rmk:KK-facts-de-equivariantize:translation-A} When $G$ is a countable discrete group and $A = C_0(G, D)$ with an action of $G$ by translation on the domain $G$, there is a natural isomorphism
			$$KK^G(C_0(G, D), B) \overset{\cong}{\longrightarrow} KK(D, B)$$ given by first applying the forgetful functor and then composing with the embedding
			$$D \cong C(\{1_{G} \}, D) \hookrightarrow C_0(G, D).$$
		\end{enumerate}
	\end{rmk}
	
	Let $EG$ denote a \emph{universal space} for free and proper $G$-actions, that is, $EG$ is a free and proper $G$-space such that any free and proper $G$-space $X$ admits a $G$-equivariant continuous map into $EG$ that is unique up to $G$-equivariant homotopy.
	Let $B G$ be the quotient of $EG$ by $G$.
	Similarly, $\underline{E}G$ denotes a \emph{universal space} for proper $G$-actions.
	These constructions are unique up to ($G$-equivariant) homotopy equivalence. By definition, there is a $G$-equivariant continuous map $EG \rightarrow \underline{E}G$, regardless of the choice of models.

	\begin{defn}\label{defn:KK-Gam-compact}
		Given a countable discrete group $G$, a locally compact, second countable, Hausdorff space $X$ with a $G$-action, a $G$-$C^*$-algebra $B$, and $i \in \mathbb{N}$, we write $KK^G_i(X, B)$ for the inductive limit of the equivariant $KK$-groups $KK^G(C_0(Z ), \Sigma  ^i B)$, where $Z$ ranges over $G$-invariant and $G$-compact subsets of $X$  directed by inclusion.  We write $KK^G_i(X)$ for $KK^G_i(X, \mathbb{C})$.
	\end{defn}

	It is clear from Bott periodicity that there is a natural isomorphism $KK^G_i(X, B) \cong KK^G_{i+2}(X, B)$. Thus we can view the index $i$ as an element of $\mathbb{Z} / 2 \mathbb{Z}$. Also note that this construction is covariant  in $X$ with respect to continuous maps. Thus there is no ambiguity in writing $KK^G_i(EG, B)$, $KK_i(BG, B)$ and $KK^G_i(\underline{E}G, B)$ for a $G$-$C^*$-algebra $B$.

	The \emph{reduced Baum-Connes assembly map} \cite{Baum-Conn_chern-character} for a countable discrete group $G$ and a $G$-$C^*$-algebra $B$ is a group homomorphism
	\[
	\mu_{\red}^B: KK^G_*(\underline{E}G, B) \rightarrow K_*(C^*_{\red}(G,B)).
	\]
	It is natural in $B$ with respect to $G$-equivariant $*$-homomorphisms or more generally with respect to taking Kasparov products, in the sense that any element $\delta \in KK^G(B,C)$ induces a commuting diagram
	\begin{equation}\label{eq:BC-assembly-natural}
		\xymatrix{
			KK^G_*(\underline{E}G, B) \ar[r]^{\mu_{\red}^B} \ar[d]^{\delta} & K_*(C^*_{\red}(G,B)) \ar[d]^{\delta \rtimes_{\operatorname{r}} G} \\
			KK^G_*(\underline{E}G, C) \ar[r]^{\mu_{\red}^C} & K_*(C^*_{\red}(G,C) )
		}
	\end{equation}
	where $\delta \rtimes_{\operatorname{r}} G\colon K_*(C^*_{\red}(G,B)) \to K_*(C^*_{\red}(G,C))$  is a homomorphism naturally induced by $\delta$.

	The case where $B = \mathbb{C}$ is of special interest. The \emph{strong Novikov conjecture} asserts that the composition
	\[
	KK^G_*(\underline{E}G) \overset{\mu_{\red}}{\longrightarrow} K_*(C^*_{\red}(G))
	\]
	is injective. And  the rational strong Novikov conjecture asserts that 
	\[
	KK^G_*(EG) \rightarrow KK^G_*(\underline{E}G) \overset{\mu_{\red}}{\longrightarrow} K_*(C^*_{\red}(G))
	\] is injective after tensoring each term by $\mathbb{Q}$. Clearly, the strong  Novikov conjecture implies that the rational strong Novikov conjecture, since  the map $KK^G_*(EG)\otimes \mathbb Q \rightarrow KK^G_*(\underline{E}G)\otimes \mathbb Q$ is always injective.  It is also known that the rational strong Novikov conjecture implies the classical Novikov conjecture on the  invariance of the higher signatures under orientation-preserving homotopy equivalences.  
	
	Even if one is primarily interested in the case where $B = \mathbb C$, it has been proven very useful to have the flexibility of a general $G$-$C^\ast$-algebra $B$ in the picture. In particular, the following key observation has proven to be very useful for studying the Baum-Connes conjecture and the (rational) strong Novikov conjecture.
	
	\begin{thm}[{cf.\,\cite[Proposition 5.11]{KasSkan}}]\label{thm:pGHT}
		For any countable discrete group $G$, and a $G$-$C^*$-algebra $B$, if $B$ is a proper $G$-algebra for some locally compact Hausdorff space $X$, then the reduced Baum-Connes assembly map
		$$
		\mu_{\red}^B : KK^G_*(\underline{E}G, B) \rightarrow K_*(C^*_{\red}(G,B))
		$$
		is an isomorphism. 
	\end{thm}
	
	This is the basis of the \emph{Dirac-dual-Dirac} method, which has brought much success to the study of the Baum-Connes assembly map. In short, the method seeks a proper $G$-$C^*$-algebra $B$ and $KK$-elements $d \in KK^G_i(B, \mathbb{C})$ and $b \in KK^G_i(\mathbb{C}, B)$. The method tells us that if the Kasparov product $\gamma_G:=b \otimes_B d$ is equal to the identity element in $KK^G(\mathbb{C},\mathbb{C})$, then the Baum-Connes assembly map for $G$ is an isomorpism and the rational strong Novikov conjecture for $G$ follows. Actually, if the image of $\gamma_G$ under the natural map   $KK^G(\mathbb{C},\mathbb{C}) \to KK(\mathbb{C},\mathbb{C})$ is equal to the identity element of $KK(\mathbb{C},\mathbb{C})$, then    the rational strong Novikov conjecture holds for $G$.

	It is not hard to see that whenever $G$ is infinite and $B$ is a proper $G$-$C^*$-algebra, there is no $G$-equivariant $*$-homomorphisms from $\Sigma = C_0(\mathbb R) $ to $B$. Thus one must look beyond $G$-equivariant $*$-homomorphisms  in order to construct a suitable element $b \in KK^G_i(\mathbb{C}, B)$. A major source of such elements is $G$-equivariant asymptotic morphisms in $E$-theory \cite{MR1065438}. Let us briefly review the definition of $E$-theory.
	Let $A$ and $B$ be separable $G$-$C^*$-algebras. The commutative semigroup $\{A,B\}_G$ of equivariant asymptotic morphisms is defined as follows.  Representing cycles are equivariant asymptotic morphisms (cf. Definition \ref{def:asymptotic}) $A\dasharrow K(\mathcal{E})$ where $\mathcal{E}$ is a separable $G$-Hilbert module over $B$, where $K(\mathcal{E})$ is the algebra of compact operators on the Hilbert module $\mathcal E$. The equivalence relation is given by homotopy, that is, two equivariant asymptotic morphisms are homotopic if there is a separable $G$-Hilbert module $\widetilde{\mathcal{E}}$ over $B\otimes C[0,1]$ and an equivariant asymptotic morphism $A\dasharrow K (\widetilde{\mathcal{E}})$ whose restrictions to the end points of the interval $[0, 1] $ are the two initial equivariant asymptotic morphisms. The sum of two equivariant asymptotic morphisms is defined by using the direct sum of  Hilbert modules. Note that $\{\Sigma  ^kA,B\}_G$ is always a group when $k\geqslant 1$.  We define the  $G$-equivariant $E$-theory of $(A, B)$ to be 
	\[  E^G(A,B) \coloneqq \{\Sigma   A \otimes  \mathcal K(H),\Sigma   B \otimes \mathcal K(H)\}_G,\] where $H = \ell^2(G) \otimes H_0$ with $H_0$ a separable Hilbert space.

	There exists  a natural transformation
	$$\eta \colon KK^G(A,B) \rightarrow E^G(A,B).$$ However,  in general there is no natural transformation the other way round. As a partial substitute, Kasparov and Higson  defined homomorphisms $$\rho \colon \{\Sigma  ^2,B\}_{G} \rightarrow KK_0^G(\mathbb{C},B)$$
	and
	$$\rho\colon  \{\Sigma  ,B\}_{G} \rightarrow KK_1^G(\mathbb{C},B),$$
	such that the compositions  $\rho\circ\eta$  are identities \cite{HK}.
	
	Let us review the construction of the homomorphisms $\eta$ and $\rho$.  
	
	\begin{defn}[{\cite[Definition 7.2]{HK}}] \label{eta}
		Let $A$ and $B$ be separable $G$-$C^*$-algebras.
		\begin{enumerate}
			\item Define a homomorphism
			$$\eta \colon KK_1^G(A,B) \rightarrow \{\Sigma A,B\}_{G} $$ as follows. View $KK_1^G(A,B) $ as the group of homotopy classes of triples $(\mathcal{E},\varphi, P)$, where  $\mathcal{E}$ is a separable $G$-Hilbert module over $B$, $\varphi\colon  A\rightarrow \mathcal{L}(\mathcal{E} ) $ is\footnote{Here $\mathcal{L}(\mathcal{E})$ is the algebra of bounded adjointable operators on the Hilbert module $\mathcal E$.} a $*$-homomorphism, and  $P\in\mathcal{L}(\mathcal{E} )$ is an operator for which $\varphi(a)(P^*-P)$, $\varphi(a)(P^2-P)$, $\varphi(a)(g(P)-P)$ and $\varphi(a)P-P\varphi(a)$ are operators in $\mathcal{K}(\mathcal{E})$, for all $a\in A$ and all $g\in G$. Here $\mathcal{K}(\mathcal{E})$ is the algebra of compact adjointable operators on the Hilbert module $\mathcal E$.
			
			Let $\{u_t\}$ be an approximate unit for $\mathcal{K}(\mathcal{E})$ which is quasicentral with respect  to  $A$, the action of $G$, and the operator $P$. We  define an asymptotic morphism
			$$\varphi\colon  \Sigma A \dasharrow \mathcal{K}(\mathcal{E})$$ by
			$$\varphi_t\colon f \otimes a \mapsto f(u_t)\varphi(a)P,~ t\geqslant 1.$$ Here the $\Sigma$ is identified with the algebra of continuous functions on the  interval $[0,1]$ vanishing at $\{0,1\}$.  Put $\eta(\mathcal{E},P)=\{\varphi_t\}$.  
			
			\item 
			Define a homomorphism
			$$\eta: KK_0^G(A,B) \rightarrow \{\Sigma^2A,B\}_{G} $$ as follows. View $KK_0^G(A,B) $ as the group of homotopy classes of pairs $(\mathcal{E},F)$, where  $\mathcal{E}$ is a separable $G$-Hilbert module over $B$, $\varphi\colon A\rightarrow \mathcal{L}(\mathcal{E} ) $ is a $*$-homomorphism, and $F\in\mathcal{L}(\mathcal{E} )$ is an operator for which $\varphi(a)(F^*F-1)$, $\varphi(a)(FF^*-1)$, $\varphi(a)(g(F)-F)$ and $\varphi(a)F-F\varphi(a)$ are operators in $\mathcal{K}(\mathcal{E})$, for all $a\in A$ and all $g\in G$.
			
			Let $\{u_t\}$ be an approximate unit for $\mathcal{K}(\mathcal{E})$ which is quasicentral with respect  to  $A$, the action of $G$, and the operator $F$, and define an asymptotic morphism
			$$\varphi: \Sigma^2A \dasharrow \mathcal{K}(\mathcal{E})$$ by
			$$\varphi_t\colon  f_1\otimes f_2 \otimes a \mapsto f_1(u_t)f_2(F)\varphi(a),~ t\geqslant 1.$$ Here the first copy of $\Sigma$ is identified with the algebra of continuous functions on the  interval $[0,1]$ vanishing at $\{0,1\}$, the second copy of $\Sigma$ is dentified with the algebra of continuous functions on unit circle vanishing at the point $1$ and $f_2(F)$ means that we take $f_2(F)$ in the Calkin algebra $\mathcal{L}(\mathcal{E})/\mathcal{K}(\mathcal{E})$ and then lift it arbitrarily to  $\mathcal{L}(\mathcal{E})$.  Put $\eta(\mathcal{E},F)=\{\varphi_t\}$.
		\end{enumerate}
	\end{defn}

	The two homomorphisms are related as follows.
	
	\begin{lem} [{\cite[Lemma 7.3]{HK}}]
		When $A = \Sigma A_1$, the homomorphism $\eta$ on $KK_1^G(A, B)$ coincides up to sign with the homomorphism $\eta$ on $KK_0^G(A_1,B)$, after $KK_0^G(A_1, B)$ and $KK_1^G(A, B)$ are identified by Bott periodicity.
	\end{lem}
	
	\begin{defn}  [{cf.\,\cite{HK}}] \label{rho}
		Let $B$ be a separable $G$-$C^*$-algebra.
		\begin{enumerate}
			\item Define a homomorphism
			$$\rho\colon  \{\Sigma  , B\}_G \longrightarrow KK_0^G(\mathbb{C}, B\otimes C_0(0,\infty))\cong KK_1^G(\mathbb{C},B)$$
			as follows. Let $\mathcal{E}$ be a separable $G$-Hilbert module over $B$ and let $\varphi_t: \Sigma  \dasharrow \mathcal{K}(\mathcal{E})$. For $t \leqslant 1$, define
			$\varphi_t=t\varphi_1$, and then extend $\varphi_t$ unitally to the unitized algebra $\widetilde{\Sigma  }$ to obtain a unital map:
			$$\varphi_t \colon  \widetilde{\Sigma  } \longrightarrow \mathcal{L}(\mathcal{E}\otimes C_0(0, \infty)).$$
			Set $V_t=\varphi_t(v)$, where $v$ is the generator of the algebra $\widetilde{\Sigma  }$, which we consider as the algebra of continuous functions on the unit circle. The element
			$V= \{V_t\}\in  \mathcal{L}(\mathcal{E}\otimes C_0(0, \infty))$ satisfies the conditions:
			$$VV^*-1, V^*V-1, g(V)-V \in \mathcal{K}(\mathcal{E}\otimes C_0(0, \infty)),  \textup{ for all } g\in G,$$
			and we define $\rho(\{\varphi_t\})$ to be the corresponding element of $KK_0^G(\mathbb{C}, B\otimes C_0(0,\infty))$.
			\item Similarly, define a homomorphism
			$$\rho: \{\Sigma  ^2, B\}_G \longrightarrow KK_1^G(\mathbb{C}, B\otimes C_0(0,\infty))\cong KK_0^G(\mathbb{C},B)$$
			as follows.  Let $\varphi_t\colon \Sigma  ^2\dasharrow \mathcal{K}(\mathcal{E})$ be an equivariant asymptotic morphism. For $t \leqslant 1$, define  $\varphi_t=t\varphi_1$, and then extend $\varphi_t$ unitally to the unitized  algebra $\widetilde{\Sigma  ^2}$ and pass to matrices to obtain a unital map
			$$\varphi_t : M_2( \widetilde{\Sigma  ^2}) \longrightarrow M_2(\mathcal{L}(\mathcal{E}\otimes C_0(0, \infty))).$$
			Set $P_t=\varphi_t(p)$, where 
			\begin{eqnarray*}  p=  \frac{1}{1+|z|^2}  \begin{pmatrix}
					1 	& z \\
					\bar z  & |z|^2 \\
				\end{pmatrix} \in M_2( \widetilde{\Sigma  ^2})
			\end{eqnarray*} is the canonical rank-one projection (i.e. the Bott generator), the algebra $ \widetilde{\Sigma  ^2}$  being identified with the algebra of continuous functions on the Riemann sphere and $z$ being the complex coordinate on the Riemann sphere. The element $P=\{P_t\}\in M_2(\mathcal{L}(\mathcal{E}\otimes C_0(0, \infty)))$
			satisfies the conditions:
			$$P^*-P, P^2-P, g(P)-P \in \mathcal{K}(\mathcal{E}\otimes C_0(0, \infty)), \forall g\in G,$$
			and we define $\rho(\{\varphi_t\})$ to be the corresponding element of $KK_1^G(\mathbb{C}, B\otimes C_0(0,\infty))$.
		\end{enumerate} 
	\end{defn}

	With the above construction, we have the following lemma from  \cite{HK}. 
	
	\begin{lem}[{\cite[Lemma 7.5]{HK}}] \label{composeid}   In the case when $A=\mathbb{C}$, the compositions $\rho\circ\eta$ give periodicity isomorphisms
		$$KK_1^G(\mathbb{C}, B)\rightarrow KK_0^G(\mathbb{C}, B\otimes C_0(0,\infty)),$$ and
		$$KK_0^G(\mathbb{C},B)\rightarrow KK_1^G(\mathbb{C},B\otimes C_0(0,\infty)).$$
	\end{lem}

	\begin{prop}\cite{KasSkan}
		Let  $A$ be a nuclear proper $G$-algebra. Then $$\eta:KK^G(A,B) \rightarrow E^G(A,B),$$ is an isomorphism.
	\end{prop}

	\section{Relative $\gamma$-reduced group $C^*$-algebras}\label{sec:gamma-groupalgebra}
	Given an pair of discrete groups $G$ and $\Gamma$, a group homomorphism $h\colon G\to \Gamma$ naturally induces  a $C^\ast$ homomorphism between the corresponding maximal group $C^\ast$-algebras $ C^\ast_{\maximal}(G)$ and $C^\ast_{\maximal}(\Gamma)$. This in turns allows us to define the relative maximal group $C^\ast$-algebra for the pair $(G, \Gamma)$, where the latter is defined to be the suspension algebra of the mapping cone algebra associated to the map $h \colon  C^\ast_{\maximal}(G) \to C^\ast_{\maximal}(\Gamma)$.  However, $h\colon G\to \Gamma$ fails to induce a $C^\ast$ homomorphism between the corresponding reduced group $C^\ast$-algebras $ C^\ast_{r}(G)$ and $C^\ast_{r}(\Gamma)$ in general. Consequently, the corresponding relative reduced group $C^\ast$-algebra is \emph{not} defined in general. In this section, we shall introduce a new relative group $C^\ast$-algebras for a pair of groups $(G, \Gamma)$, under the assumption that the group $G$ has a $\gamma$ element (in the sense of Kasparov). The class of  groups that have a $\gamma$ element is rather large. Notably, Tu showed that if a group $G$ is coarsely embeddable into Hilbert space, then $G$ has a $\gamma$ element \cite{Tu_gamma}. For this reason, the new relative group $C^\ast$-algebras constructed in this section will be called relative $\gamma$-reduced group $C^\ast$-algebras. They constitute a key ingredient in the proof of our main theorem in this paper.

	Let us first recall what it means for a discrete group to have a $\gamma$ element in the sense of Kasparov. 
	
	\begin{defn}\label{def:gammaelement}
		Assume that $G$ is a countable discrete group. We say $G$ has a $\gamma$-element if  there exists a proper $G$-$C^*$-algebra  $\mathcal{A}$, and elements $b\in KK_i^G(\mathbb{C}, \mathcal{A})$  and  $d\in KK_i^G(\mathcal{A},\mathbb{C})$, for $i\in \mathbb{Z}/2\mathbb{Z}$  such that
		the element 
		\[  \gamma_{G} \coloneqq d\otimes_\mathcal{A}b\in KK^G(\mathbb{C},\mathbb{C}) \]
		becomes $1 \in KK^F(\mathbb{C},\mathbb{C})$ under the forgetful map $KK^G(\mathbb{C},\mathbb{C})\rightarrow KK^F(\mathbb{C},\mathbb{C})$,  for all finite subgroups $F\subset G$.
	\end{defn}

    If $G$ has a $\gamma$ element $\gamma_G$ in the sense of the above definition, then $\gamma_G$ naturally induces a homomorphism on the $K$-homology $K_i^G(\underline EG) = K_i(C_L^\ast(\underline EG)^G)$. By using Meyer-Vietoris sequences, a standard cutting-and-pasting argument shows that $(\gamma_G)_\ast\colon K_i^G(\underline EG) \to K_i^G(\underline EG)$ is equal to the identity map. 
	
	We also have the following analogue of $\gamma$ elements when we work with rational coefficients. 
	
	\begin{defn}\label{def:rationalgamma}
		Assume that $G$ is a countable discrete group. We say $G$ has a \emph{rational} $\gamma$-element if  there exists a proper $G$-$C^*$-algebra  $\mathcal{A}$, and elements $b\in KK_i^G(\mathbb{C}, \mathcal{A})$  and  $d\in KK_i^G(\mathcal{A},\mathbb{C})$, for $i\in \mathbb{Z}/2\mathbb{Z}$  such that
		the element 
		\[  \gamma_{G} \coloneqq d\otimes_\mathcal{A}b\in KK^G(\mathbb{C},\mathbb{C}) \]
		becomes $1 \in KK(\mathbb{C},\mathbb{C})$ under the forgetful map $KK^G(\mathbb{C},\mathbb{C})\rightarrow KK(\mathbb{C},\mathbb{C})$.
	\end{defn}

	Now let us assume $G$ is a countable discrete group with a  $\gamma$ element. Let $b\in KK_1^G(\mathbb{C}, \mathcal{A})$  and  $d\in KK_1^G(\mathcal{A},\mathbb{C})$ be as above.   By the construction from  Definition \ref{eta},
	we have 
	\[ \eta(b)\in \{\Sigma  ,\mathcal{A}\}_{G} \textup{ and } \eta(d)\in \{\Sigma  \mathcal{A}, \mathbb{C}\}_{G}. \] Since $\mathcal{A}$ is a  proper $G$-$C^*$-algebra, according to the equivariant stabilization theorem,  the $G$-Hilbert module over $\mathcal{A}$ can be taken as $\mathcal{A}\otimes H$, where $H=\oplus_{n=1}^\infty \ell^2(G).$ Hence $\eta(b)$ can be written as a $G$-equivariant asymptotic morphism $$\eta(b)\colon \Sigma  \dasharrow  \mathcal{A}\otimes \mathcal{K}(H).$$ Similarly, $\eta(d)$ is a $G$-equivariant asymptotic morphism 
	\[ \eta(d)\colon \Sigma  \mathcal{A}\dasharrow \mathcal{K}(H).\] Let us define 
	\begin{equation}\label{eq:mapbeta}
		\beta \coloneqq \id \otimes \eta(b) \colon  \Sigma^2 = \Sigma\otimes \Sigma  \dasharrow \Sigma\mathcal{A}\otimes \mathcal{K}(H)
	\end{equation}
	and  
	\begin{equation}\label{eq:mapalpha}
		\alpha\coloneqq \eta(d)\otimes \id \colon  \Sigma  \mathcal{A}\otimes \mathcal{K}(H)\dasharrow \mathcal{K}(H)\otimes \mathcal{K}(H)\cong \mathcal{K}(H).
	\end{equation}  
	Furthermore, we have   $$\eta(\gamma_{_G})\colon \Sigma  ^2\dasharrow \mathcal{K}(H)$$ which can also be viewed as the composition of   the  asymptotic morphisms $\alpha$ and $\beta$.

	If no confusion is likely to arise, we shall also use $\gamma_{_G}$ to denote  the asymptotic morphism $\eta(\gamma_{_G})\colon \Sigma  ^2\dasharrow \mathcal{K}(H) $. For any subgroup $F$ of $G$, we write    $\gamma_{_F}\colon \Sigma  ^2\dasharrow \mathcal{K}(H)$ to denote  the asymptotic morphism $\gamma_G$ but viewed as  $F$-equivariant asymptotic morphism.

	\begin{rmk}\label{representation-of-G}
		In the above construction,  $G$ acts on the Hilbert space $H = \oplus_{n=1}^\infty \ell^2(G)$ via the left regular representation of $G$ on each $\ell^2(G)$. If  we denote this action of $G$ on $H$ by $U_g$ for each $g\in G$,  then the induced action of $G$ on $\mathcal K(H)$ is given by $$g\cdot K=U_gKU_g^*,$$ for any $K\in\mathcal{K}(H)$ and  any $g\in G$. 
	\end{rmk}

	The following lemma is a key ingredient in our construction of relative $\gamma$-reduced group $C^*$-algebras. 
	
	\begin{lem}\label{key-lemma}
		Suppose $G$ has a  $\gamma$-element.
		Given  a group homomorphism $h\colon G\rightarrow\Gamma$, there exists a natural asymptotic morphism 
		\begin{equation}\label{eq:homogamma}
			h_{\gamma}\colon  \Sigma^2C_{\red}^*(G)\dashrightarrow C_{\red}^*(\Gamma)\otimes  \mathcal{K}(H)
		\end{equation}
		such that 
		$$h_{\gamma,t}=h_{t} \textup{ on } C_c(G, \Sigma  ^2),$$
		where for each $t\in[1,+\infty)$, $h_{t}$ is a map  $C_c(G, \Sigma  ^2)\longrightarrow C_c(\Gamma, \mathcal{K}(H))$ given by
		$$h_{t}(\sum f_g\cdot g)=\sum (\gamma_{_{G,t}}(f_g) \cdot U_g )\cdot h(g).$$
	\end{lem}
	
	\begin{proof} 
		Since $G$ has a (rational) $\gamma$-element, the above discussion shows that 
		we have $G$-equivariant asymptotic morphisms 
		$$\beta\colon  \Sigma^2  \dasharrow  \Sigma\mathcal{A}\otimes \mathcal{K}(H),$$ 
		and 
		$$\alpha\colon  \Sigma  \mathcal{A}\otimes \mathcal{K}(H)\dasharrow \mathcal{K}(H),$$
		such that their composition is the  asymptotic morphism 
		$$\gamma_{_G}\colon \Sigma  ^2\dasharrow \mathcal{K}(H).$$ 
		
		Consider the group homomorphism 
		$\widetilde h\colon  G\rightarrow G\times\Gamma$ defined by $\widetilde h(g)=(g,h(g))$, for any $g\in G$. Since $\widetilde h$ is injective, it induces a $C^\ast$-algebra homomorphism  $$\widetilde h \colon C_{\red}^*(G, \Sigma  ^2)\longrightarrow C_{\red}^*(G\times\Gamma, \Sigma  ^2).$$

		Let  $G\times\Gamma$ act on   $\Sigma  \mathcal{A}\otimes \mathcal{K}(H)$ by  $$(g,g^\prime)\cdot a:=g\cdot a,$$ for any $(g,g^\prime)\in G\times\Gamma$ and $a\in \Sigma  \mathcal{A}\otimes \mathcal{K}(H)$. Note that we have the  natural isomorphism
		$$\zeta\colon  C_{\red}^*(G\times\Gamma, \Sigma  \mathcal{A}\otimes \mathcal{K}(H))
		\rightarrow C_{\red}^*(G, \Sigma  \mathcal{A}\otimes \mathcal{K}(H))\otimes C_{\red}^*(\Gamma).$$
		
		Since $\Sigma  \mathcal{A}\otimes \mathcal{K}(H)$ is a proper $G$-$C^*$-algebra,  the maximal crossed product of $\Sigma  \mathcal{A}\otimes \mathcal{K}(H)$ by $G$ coincides with the corresponding reduced crossed product:
		$$\theta \colon  C_{\maximal}^*(G, \Sigma  \mathcal{A}\otimes \mathcal{K}(H))
		\stackrel{\cong}{\longrightarrow}C_{\red}^*(G, \Sigma  \mathcal{A}\otimes \mathcal{K}(H)).$$
		
		Denote by
		$$\pi\colon C_{\maximal}^*(G,\mathcal{K}(H))\rightarrow \mathcal{K}(H)$$ the $*$-homomorphism  induced by  the trivial group homomorphism
		$\pi\colon G\rightarrow \{e\}$. More specifically, we have 
		\[  \pi(\sum K_g g) = \sum K_g U_g  \]
		for all $K_g \in \mathcal K(H)$ and $g\in G$, where $U_g$ is left translation of $g$ on $H = \oplus_{n=1}^\infty \ell^2(G)$ (cf. Remark \ref{representation-of-G}).

		The $G$-equivariant asymptotic morphisms $\alpha$ and $\beta$ naturally extend to give  asymptotic morphisms on the following crossed products:
		$$\beta\colon C_{\red}^*(G\times\Gamma, \Sigma^2)
		\dashrightarrow  C_{\red}^*(G\times\Gamma, \Sigma  \mathcal{A}\otimes \mathcal{K}(H))$$
		and
		$$\alpha\colon C_{\maximal}^*(G, \Sigma  \mathcal{A}\otimes \mathcal{K}(H))
		\dashrightarrow C_{\maximal}^*(G,\mathcal{K}(H)).$$

		Now let $h_{\gamma}: \Sigma^2C_{\red}^*(G)\dashrightarrow C_{\red}^*(\Gamma)\otimes  \mathcal{K}(H)$ be the asymptotic morphism obtained as the following  composition
		\[\begin{tikzcd}
			\Sigma^2C_{\red}^*(G) & C_{\red}^*(G\times\Gamma, \Sigma^2) & C_{\red}^*(G\times\Gamma, \Sigma  \mathcal{A}\otimes \mathcal{K}(H)) \\ 
			\\
			\mathcal{K}(H) \otimes C_r^\ast(\Gamma) & C_{\maximal}^*(G,\mathcal{K}(H))\otimes C_r^\ast(\Gamma) & C_{\maximal}^*(G, \Sigma  \mathcal{A}\otimes \mathcal{K}(H)) \otimes C_r^\ast(\Gamma)
			\arrow["\widetilde h", from=1-1, to=1-2]
			\arrow["\beta", dashed, from=1-2, to=1-3]
			\arrow["\cong", from=1-3, to=3-3]
			\arrow["\pi\otimes \id "', from=3-2, to=3-1]
			\arrow["\alpha\otimes \id "', dashed, from=3-3, to=3-2]
		\end{tikzcd}\]
		It is clear from the above construction that $$h_{\gamma,t}=h_{t} \textup{ on } C_c(G, \Sigma  ^2),$$
		where for each $t\in[1,+\infty)$, $h_{t}$ is a map  $C_c(G, \Sigma  ^2)\longrightarrow C_c(\Gamma, \mathcal{K}(H))$ given by
		$$h_{t}(\sum f_g\cdot g)=\sum (\gamma_{_{G,t}}(f_g) \cdot U_g )\cdot h(g).$$
		This finishes the proof. 
	\end{proof}

	The asymptotic morphism $h_{\gamma}$ from the above lemma naturally induces  a $C^*$ homomorphism  (also denoted by $h_{\gamma}$)
	\begin{equation}\label{eq:asymp-reduced}
		h_{\gamma}\colon  C_{\red}^*(G, \Sigma  ^2)\rightarrow  \asymalg(C_{\red}^*(\Gamma)\otimes \mathcal{K}(H))
	\end{equation}
	(cf. the discussion after Definition \ref{uniformly-asymptotic-algebra}).

	\begin{defn}\label{relative-gamma-reduced-algebras}
		Let $G$ and $\Gamma$ be countable discrete groups. Suppose $G$ has a (rational)  $\gamma$-element. For any group homomorphism $h\colon G\rightarrow\Gamma$,  we define the   \emph{relative $\gamma$-reduced group $C^*$-algebra} $C_{{\gamma}}^*(G,\Gamma)$ to be the mapping cone associated to the map
		$h_{\gamma}$ in line \eqref{eq:asymp-reduced}, that is,  $C_{{\gamma}}^*(G,\Gamma)$ is the following $C^\ast$ algebra 
		\[  \{(a, f)\in  C_{\red}^*(G, \Sigma  ^2) \oplus C_0([0, 1),  \asymalg(C_{\red}^*(\Gamma)\otimes \mathcal{K}(H))) \mid h_\gamma(a) = f(0)\}. \]
	\end{defn}

	\section{Relative $K$-homology and relative Baum-Connes assembly map}\label{sec:relativetheory}
	
	In this section, we  review the construction of relative $K$-homology  and relative Baum-Connes assembly map (cf. \cite{MR4704778}).

	\subsection{Relative K-homology}\label{sec:relativeK}

	\begin{defn}  Let $\Gamma$ be a finitely generated group with a word length metric $d$. Let $s > 0$. The \textit{Rips complex} of $\Gamma$ at scale $s$, denoted by $P_s(\Gamma)$, is the simplicial complex with the vertex set $\Gamma$ such that a subset $\{\gamma_0,\cdots,\gamma_n\}$ of $\Gamma$ spans a simplex if and only if $d(\gamma_i,\gamma_j)\leqslant s$ for all $i,j$.
	\end{defn}
	
	Each Rips complex $P_s(\Gamma)$ is equipped with the spherical metric. Recall that the spherical metric is the maximal metric whose restriction to each simplex $\left\{\sum_{i=0}^n c_i t_i\right\} \subset P_s(\Gamma)$ is the metric obtained by identifying this simplex with the upper hemisphere $S^n_+=\left\{(t_0,t_1,\cdots,t_n): \sum_{i=0}^n t^2_i=1, t_i\geq 0, ~\forall 0\leq i \leq n\right\}$ by
	$$
	\left(c_0,c_1,\cdots, c_n\right) \mapsto\left(\frac{c_0}{\sqrt{\sum_i c^2_i}}, \frac{c_1}{\sqrt{\sum_i c^2_i}},\cdots, \frac{c_n}{\sqrt{\sum_i c^2_i}}\right)
	$$
	where $S^n_+\subset \mathbb{R}^{n+1}$ is endowed with the standard round metric of the upper hemisphere. 
	
	Note that $\Gamma$ naturally acts on  each $P_s(\Gamma)$. More precisely, for each element $x=\sum_{\gamma\in\Gamma}t_\gamma\gamma\in P_s(\Gamma)$ and $g \in \Gamma$, we have 
	$$g \cdot \big(\sum_{\gamma\in\Gamma}t_\gamma\gamma\big)=\sum_{\gamma\in\Gamma}t_\gamma g\gamma.$$
	It is obvious that this $\Gamma$-action is proper.
	
	Suppose $G$ and $\Gamma$ are finitely generated groups, and $h \colon G \to \Gamma$ is a group homomorphism. Assume that $S \subset G$ is a finite symmetric generating set of $G$, that is, $S$ generates $G$ and  $g^{-1} \in S$ for each $g \in S$. There is a left invariant word length metric $d_G$ on $G$ naturlly associated to the generating subset $S$. Furthermore, if we choose  a finite symmetric generating set $S' \subset \Gamma$ that contains $h(S)$, then the left invariant metric $d_{\Gamma}$ on $\Gamma$ determined by $S'$ satisfies that  $d_{\Gamma}(h(g_1),h(g_2))\leq d_G(g_1,g_2)$ for any $g_1,g_2$ in $G$. For each $s>0$, 
	the group homomorphism  $h\colon G \to \Gamma$  extends to a continuous map (also defined by $h$)
	$$h\colon  P_s(G) \to P_s(\Gamma)$$
	by
	$$h(\sum_{\gamma\in\Gamma}t_\gamma\gamma)=\sum_{\gamma\in\Gamma}t_\gamma h(\gamma)$$
	for each $\sum_{\gamma\in\Gamma}t_\gamma\gamma \in P_s(\Gamma)$. Note that
	$$d_{P_s(\Gamma)}(h(x),h(y))\leq d_{P_s(G)}(x, y)$$
	for all $s>0$ and all $x, y \in P_s(G)$.  
	It follows that $h\colon G\to \Gamma$ induces a $C^\ast$ homomorphism
	\begin{equation}\label{eq:localhomo}
		h_{\maximal, L}\colon C_{\maximal, L}^*(P_sG)^G
		\overset{}{\longrightarrow} C_{L}^* (P_{s}(\Gamma))^{\Gamma}.
	\end{equation}
	See Section \ref{sec:Roe-local} for the precise definition of localization algebras $C_{L}^*(P_sG)^G$ and $C_{L}^* (P_{s}(\Gamma))^{\Gamma}$. 
	
	\begin{defn}\label{def:relativeK}
		We define  $C_{L}^{*}(P_{s}G, P_{s}\Gamma)^{G,\Gamma}$ to be 
		the mapping cone of the map 
		$$ h_{\maximal, L}\colon C_{\maximal, L}^*(P_sG)^G
		\to  C_{L}^* (P_{s}\Gamma)^{\Gamma}.$$
		We define 
		$$K_{i}^{G,\Gamma}(\underline{E}G,\underline{E}\Gamma):=\lim\limits_{s\rightarrow \infty}K_i(C_{L}^{*}(P_{s}G, P_{s}\Gamma)^{G,\Gamma}).
		$$
	\end{defn}
	
	\begin{rmk}
		Even though we have used both the maximal localization algebra and the reduced localization algebra in the above definition of the relative localization algebra $C_{L}^{*}(P_{s}G, P_{s}\Gamma)^{G,\Gamma}$, the $K$-theory of $C_{L}^{*}(P_{s}G, P_{s}\Gamma)^{G,\Gamma}$ coincides with the usual   $(G, \Gamma)$-equivariant relative $K$-homology of the pair $(P_sG, P_s\Gamma)$. 
	\end{rmk}

	Let us conclude this subsection with the following lemma, the proof of which is elementary. 
	
	\begin{lem}\label{product-structure}
		Let $G$ and $\Gamma$ be countable discrete groups, and let $X$ and $Y$ be metric spaces equipped with proper actions of $G$ and $\Gamma$, respectively. Then for any given $G$-algebra $A$ and $\Gamma$-algebra $B$, there exist natural isomorphisms
		$$C_{\red}^*(X\times Y,A \otimes B)^{G\times \Gamma} \overset{\cong}{\longrightarrow}
		C_{\red}^*(X,A)^G\otimes  C_{\red}^*(Y,B)^\Gamma,  $$
		and 
		$$C_{L}^*(X\times Y,A \otimes B)^{G\times \Gamma}\overset{\cong}{\longrightarrow}
		C_{L}^*(X,A)^G\otimes  C_{L}^*(Y,B)^{\Gamma}. $$
	\end{lem}

	\subsection{Relative Baum-Connes assembly map}\label{sec:relativeBC}
	We defined  a relative $\gamma$-reduced group $C^\ast$-algebra for each group homomorphism $h\colon G\to \Gamma$, under the assumption that $G$ has a (rational) $\gamma$ element in Section \ref{sec:gamma-groupalgebra}. In order to formulate the relative Baum-Connes assembly map for the context of these new relative $\gamma$-reduced group $C^\ast$-algebras, we shall first give an alternative  description of the relative $K$-homology $K_{i}^{G,\Gamma}(\underline{E}G,\underline{E}\Gamma)$.

	Let us first consider the construction at the level of Roe algebras. Suppose $G$ and $\Gamma$ are finitely generated groups, and $h \colon G \to \Gamma$ is a group homomorphism. By the same discussion as in Section \ref{sec:relativeK},  we have the map 
	$h\colon  P_s(G) \to P_s(\Gamma).$
	Let us choose a countable dense $G$-invariant subset $X$ of $P_s(G)$ and a countable dense $\Gamma$-invariant subset $Y$ of $P_s(\Gamma)$ such that  $h(X)$ is a subset of $Y$. Suppose $G$ has a rational $\gamma$-element. We define a family of maps  
	$$h_{t}\colon \mathbb{C}[P_sG, \Sigma  ^2]^G \rightarrow \mathbb{C}[P_s\Gamma, \mathcal{K}(H)]^\Gamma $$
	with  $t\in [1,+\infty)$ 
	as follows (cf. Lemma \ref{key-lemma}). Let $\gamma_{_G}\colon \Sigma  ^2\dasharrow \mathcal{K}(H) $ be the asymptotic morphism from Section \ref{sec:gamma-groupalgebra}. 
	An  element $A\in \mathbb{C}[P_sG, \Sigma  ^2]^G$ can be viewed as a $G$-invariant map $X\times X\rightarrow \Sigma  ^2\otimes \mathcal{K}(H)$, we define $h_t(A)$ to be the map $h_t(A)\colon Y\times Y\rightarrow \mathcal{K}(H)$ given by 
	$$
	h_t(A)(z,w):=\sum\limits_{g \in \ker(h)}\gamma_{_{G,t}}(T_{gx,y})\cdot U_g
	$$
	if there exists $(x,y)\in X\times X$ such that $h(x)=z$ and $h(y)=w$; otherwise, set  $h_t(A)(z,w):=0$. Note that in the former case, the value $h_t(A)(z,w)$ does not depend on the choice of $(x, y)$ as long as we have $h(x)=z$ and $h(y)=w$.   It is straightforward  to check that $h_t(A)$ is a $\Gamma$-invariant and has finite propagation,  thus lies in  $\mathbb{C}[P_s\Gamma, \mathcal{K}(H)]^\Gamma$.
	
	We have the following obvious analogue of  Lemma \ref{key-lemma}.
	
	\begin{lem}
		Suppose $G$ has a rational $\gamma$-element. Given  a group homomorphism $h\colon G\rightarrow\Gamma$, for each $s>0$, there exists an  asymtotic morphism
		\begin{equation}\label{eq:homogammaRoe}
			h_{\gamma}\colon C_{\red}^*(P_sG, \Sigma  ^2)^G \overset{}{\dashrightarrow}  C_{\red}^*(P_{s}\Gamma,\mathcal{K}(H))^\Gamma
		\end{equation}
		such that 
		$$h_{\gamma,t}=h_{t} \textup{ on } \mathbb{C}[P_sG, \Sigma  ^2]^G.$$
	\end{lem}

	The definition of the map $h_{t}\colon \mathbb{C}[P_sG, \Sigma  ^2]^G \rightarrow \mathbb{C}[P_s\Gamma, \mathcal{K}(H)]^\Gamma $ extends in an obvious manner to a map 
	$$h^L_{t}\colon \mathbb{C}_L[P_sG, \Sigma  ^2]^G \rightarrow \mathbb{C}_L[P_s\Gamma, \mathcal{K}(H)]^\Gamma.$$
	We  have the following analogue of Lemma \ref{key-lemma} for the correpsonding localization algebras. 
	
	\begin{prop} \label{relative-localization}
		Suppose $G$ has a rational $\gamma$-element. Given  a group homomorphism $h\colon G\rightarrow\Gamma$, for each $s>0$, 
		there is a natural  asymptotic morphism  
		\begin{equation}\label{eq:localgamma}
			h^L_{\gamma}\colon C_{L}^*(P_sG, \Sigma  ^2)^G \overset{}{\dashrightarrow}  C_{L}^*(P_{s}\Gamma,\mathcal{K}(H))^\Gamma
		\end{equation}
		such that 
		$$h^L_{\gamma,t}=h^L_{t} \textup{ on } \mathbb{C}_L[P_sG, \Sigma  ^2]^G$$ and  the following diagram commutes 
		\begin{equation*}
			\begin{tikzcd}
				C_{L}^*(P_sG, \Sigma  ^2)^G\arrow[r, "h^L_{\gamma,t}"]\arrow[d,"\ev"]& C_{L}^*(P_{s}\Gamma,\mathcal{K}(H))^\Gamma\arrow[d,"\ev"] \\
				C_{\red}^*(P_sG, \Sigma  ^2)^G\arrow[r, "h_{\gamma,t}"]& C_{\red}^*(P_{s}\Gamma,\mathcal{K}(H))^\Gamma
			\end{tikzcd}
		\end{equation*}
		for each $t\in [1,+\infty)$, where $\ev$ is the evaluation map at zero.  
	\end{prop}

	The above asymptotic morphisms induce the following $C^\ast$ homomorphsims: 
	$$
	h_{\gamma}\colon C_{\red}^*(P_s(G), \Sigma  ^2)^G \overset{}{\longrightarrow}  \asymalg( C_{\red}^*(P_{s}(\Gamma),\mathcal{K}(H))^\Gamma)
	$$
	and
	$$
	h^L_{\gamma}:C_{L}^*(P_s(G), \Sigma  ^2)^G \overset{}{\longrightarrow}  \asymalg( C_{L}^*(P_{s}(\Gamma),\mathcal{K}(H))^\Gamma).
	$$
	
	\begin{defn}
		Suppose $G$ has a (rational) $\gamma$-element. For any group homomorphism
		$h\colon G\rightarrow\Gamma$, we  define $C_{r,\gamma}^{*}(P_sG,P_{s}\Gamma)^{G, \Gamma}$ to be the mapping cone of  \[ C_{\red}^*(P_s(G), \Sigma  ^2)^G \overset{}{\longrightarrow}  \asymalg( C_{\red}^*(P_{s}(\Gamma),\mathcal{K}(H))^\Gamma), \] and similarly  $C_{L, \gamma}^{*}(P_sG,P_{s}\Gamma)^{G,\Gamma}$ to be the mapping cone of $$
		h^L_{\gamma}:C_{L}^*(P_s(G), \Sigma  ^2)^G \overset{}{\longrightarrow}  \asymalg( C_{L}^*(P_{s}(\Gamma),\mathcal{K}(H))^\Gamma).
		$$
	\end{defn}

	It is not difficult to see that $C_{r,\gamma}^{*}(P_sG,P_{s}\Gamma)^{G, \Gamma}$ is isomorphic to $C_{\gamma}^{*}(G, \Gamma)\otimes \mathcal{K}(H)$, where  $C_{\gamma}^{*}(G, \Gamma)$ is the relative $\gamma$-reduced group $C^\ast$-algebra introduced in Definition \ref{relative-gamma-reduced-algebras}. Let us summarize this as the following proposition. 
	
	\begin{prop}\label{prop:Roeiso}
		For any $s>0$, $$C_{r,\gamma}^{*}(P_sG,P_{s}\Gamma)^{G, \Gamma}\cong C_{\gamma}^{*}(G, \Gamma)\otimes \mathcal{K}(H).$$
	\end{prop}
	\begin{proof}
		Recall that $C^\ast_r(P_sG)^G \cong C^\ast_r(G)\otimes \mathcal K(H)$.   Now the proposition  follows from the construction of the map 
		$h_{t}\colon \mathbb{C}[P_sG, \Sigma  ^2]^G \rightarrow \mathbb{C}[P_s\Gamma, \mathcal{K}(H)]^\Gamma $
		and the relative $\gamma$-reduced $C^*$-algebra $C_{\gamma}^{*}(G, \Gamma)$.
	\end{proof}
	Now consider the evaluation map
	\begin{eqnarray*}
		\ev\colon  C^*_{L,\gamma}(P_s(G), P_s(\Gamma))^{G, \Gamma} \to C_{\red,\gamma}^*(P_s(G), P_s(\Gamma))^{G, \Gamma}.
	\end{eqnarray*}
	For any $s_1<s_2$, the following diagram commutes 
	\begin{equation*}
		\begin{tikzcd}
			K_*(C^*_{L,\gamma}(P_{s_1}G, P_{s_1}\Gamma)^{G, \Gamma} ) \ar{r}{{\ev}_*} \ar{d}
			&  K_*(C^*_{\red,\gamma}(P_{s_1}G, P_{s_1}\Gamma)^{G, \Gamma} )\ar{r}{\cong} \ar{d}& K_*(C_{\gamma}^*(G,\Gamma)) \ar[equal]{d}\\
			K_*(C^*_{L,\gamma}(P_{s_2}G, P_{s_2}\Gamma)^{G, \Gamma} ) \ar{r}{{\ev}_*}&
			K_*(C^*_{\red,\gamma}(P_{s_2}G, P_{s_2}\Gamma)^{G, \Gamma} ) \ar{r}{\cong}& K_*(C_{\gamma}^*(G,\Gamma))
		\end{tikzcd}
	\end{equation*}

	\begin{defn} For a finitely generated group $G$ with a (rational) $\gamma$-element, we define 
		$$K_{i,\gamma}^{G,\Gamma}(\underline{E}G,\underline{E}\Gamma):=\lim_{s\rightarrow\infty}K_i(C_{L,\gamma}^{*}(P_sG,P_{s}\Gamma)^{G,\Gamma}).$$ 
	\end{defn}

	Next we shall show that the equivariant relative $K$-homology group $K_{i,\gamma}^{G,\Gamma}(\underline{E}G,\underline{E}\Gamma)$ defined above coincides with the classical equivariant relative $K$-homology group given in Definition \ref{def:relativeK}.

	Let us first fix some notation. Given a uniformly continuous  asymptotic morphism  $h\colon A\dasharrow B$ between  separable $C^*$-algebras. Let us denote the induced homomorphism $A\to  \asymalg(B)$ still by $h$. We denote by 
	$C_h$ the mapping cone associated to $h\colon A\to  \asymalg(B)$.

	\begin{lem}\label{lm:exact}
		Given a uniformly continuous asymptotic morphism  $h\colon A\dasharrow B$, there is a natural long exact sequence of $K$-theory groups,
		$$
		K_{0}( \Sigma  A)\xrightarrow{h_*} K_{0}(\Sigma  B)\xrightarrow{\imath_*} K_{0}( C_h)\xrightarrow{r_*} K_{0}(A)\xrightarrow{h_*} K_{0}(B),
		$$
		where $\imath\colon \Sigma\otimes \asymalg B\hookrightarrow C_h$ is the natural injection  and $r\colon C_h\rightarrow A$ is the restriction on $A$.
	\end{lem}
	
	\begin{proof}
		Consider the short exact sequence: $$0\rightarrow \Sigma\otimes\asymalg(B) \xrightarrow{\imath} C_h \xrightarrow{r}  A\rightarrow 0.$$
		which induces the following long exact sequence in $K$-theory: 
		$$
		K_{0}(\Sigma  A)\xrightarrow{h_*} K_{0}(\Sigma   \otimes \asymalg(B))\xrightarrow{\imath_*} K_{0}( C_h)\xrightarrow{r_*} K_{0}(A)\xrightarrow{h_*}K_{0}(\asymalg(B)).
		$$
		The lemma now follows by the following commutative diagram: 
		\begin{equation*}
			\begin{tikzcd}
				K_{*}(A)\arrow[dr, swap, "(h_t)_*"] \arrow{r}{h_*} & K_{*}(\asymalg(B))\\
				& K_*(B)\arrow{u}{\theta_*}
			\end{tikzcd}
		\end{equation*}
		where $\theta_*$ is an isomorphism from Lemma \ref{reversetrick}.
	\end{proof}

	From the discussion above, we have a  commutative diagram (Figure \ref{fig:khomologymap}) for each $s>0$, which shows that the composition of asymptotic morphisms $\alpha\circ \beta$ induces a natural asymptotic morphism from  $C_{L,\gamma}^{*}(P_sG,P_{s}\Gamma)^{G, \Gamma}$ to $C^{*}_{L}(P_{s}G,P_{s}\Gamma)^{G,\Gamma} \otimes \mathcal{K}(H)$.
	\begin{figure}
		\adjustbox{scale=0.9,center}{\begin{tikzcd}
				C_{L}^*(P_sG, \Sigma^2)^G &  C_{L}^*(P_sG, \Sigma  \mathcal{A}\otimes \mathcal{K}(H))^G & C_{ \maximal, L}^*(P_sG, \Sigma  \mathcal{A}\otimes \mathcal{K}(H))^G  \\
				& C_{L}^*(P_sG\times P_s\Gamma,\Sigma  ^2)^{G\times\Gamma}  &    \\
				& C_{L}^*(P_{s}G\times P_s\Gamma,\Sigma  \mathcal{A}\otimes \mathcal{K}(H))^{G\times\Gamma} \\
				& C_{L}^*(P_{s}G,\Sigma  \mathcal{A}\otimes \mathcal{K}(H))^G\otimes  C_{L}^*(P_{s}\Gamma)^\Gamma \\
				& C_{\maximal, L}^*(P_{s}G,\Sigma  \mathcal{A}\otimes \mathcal{K}(H))^G\otimes  C_{L}^*(P_{s}\Gamma)^\Gamma \\
				& C_{\maximal, L}^*(P_{s}G,\mathcal{K}(H))^G\otimes C_{L}^*(P_{s}\Gamma)^\Gamma  \\
				& C_{L}^*(pt, \mathcal{K}(H))\otimes C_{L}^*(P_s\Gamma)^\Gamma \\
				& \mathcal{K}(H)\otimes C_{L}^*(P_s\Gamma)^\Gamma & C_{\maximal, L}^*(P_{s}G,\mathcal{K}(H))^G
				\arrow["\cong", from=3-2, to=4-2]
				\arrow["\cong", from=4-2, to=5-2]
				\arrow["{\alpha\otimes \id}", dashed,  from=5-2, to=6-2]
				\arrow["{\pi\otimes \id}", from=6-2, to=7-2]
				\arrow["{\ev \otimes \id }", from=7-2, to=8-2]
				%\arrow["h_{\alpha}^{L, \mathcal{A}}", bend left=40, dashed, from=1-3, to=8-2]
				\arrow["\beta", dashed, from=1-1, to=1-2]
				\arrow["\cong", from=1-2, to=1-3]
				\arrow["{\widetilde h}"', from=1-1, to=2-2]
				\arrow["h^L_\gamma"', bend right=40, dashed, from=1-1, to=8-2]
				\arrow["\beta", dashed, from=2-2, to=3-2]
				%\arrow["{\widetilde h}"', bend left=20, from=1-3, to=3-2]
				%\arrow[from=2-3, to=3-2]
				\arrow["\alpha", dashed, from=1-3, to=8-3]
				\arrow["h_{\maximal, L}"', from=8-3, to=8-2]
			\end{tikzcd}
		}
		\caption{A commutative diagram}
		\label{fig:khomologymap}
	\end{figure}
	More precisely, recall that $C_{L,\gamma}^{*}(P_sG,P_{s}\Gamma)^{G, \Gamma}$ is the mapping cone of 
	\[ h^L_\gamma\colon C_{L}^*(P_sG, \Sigma^2)^G \to \asymalg(C_{L}^*(P_s\Gamma)^\Gamma\otimes \mathcal{K}(H)). \]

	Now the asymptotic morphism $\alpha\circ \beta \colon C_{L}^*(P_sG, \Sigma^2)^G  \dasharrow C_{\maximal, L}^*(P_{s}G,\mathcal{K}(H))^G $ induces a $C^\ast$ homomorphism 
	\[  \varPhi_{\alpha\circ \beta} \colon C_{L}^*(P_sG, \Sigma^2)^G  \to \asymalg C_{\maximal, L}^*(P_{s}G,\mathcal{K}(H))^G \]
	which fits into the following commutative diagram: 
	\[\begin{tikzcd}
		& C_{L}^*(P_sG, \Sigma^2)^G &  \\
		\asymalg C_{\maximal, L}^*(P_{s}G,\mathcal{K}(H))^G &	 & \asymalg(C_{L}^*(P_s\Gamma)^\Gamma\otimes \mathcal{K}(H)). 
		\arrow["\varPhi_{\alpha\circ \beta}"', from=1-2, to=2-1]
		\arrow["h^L_\gamma", from=1-2, to=2-3]
		\arrow["h_{\maximal, L}", from=2-1, to=2-3]
	\end{tikzcd}\]
	The $C^*$ homomorphism $h_{\maximal, L}\colon C_{\maximal, L}^*(P_{s}G,\mathcal{K}(H))^G \to   C_{L}^*(P_s\Gamma)^\Gamma \otimes \mathcal{K}(H)$ naturally induces a $C^*$ homomorphism
	\[ h_{\maximal, L}\colon \asymalg C_{\maximal, L}^*(P_{s}G,\mathcal{K}(H))^G \to  \asymalg( C_{L}^*(P_s\Gamma)^\Gamma \otimes \mathcal{K}(H)),  \]
	whose corresponding mapping cone will be denote by $C_{P_s, h}$.
	It follows from the above discussion that $\varPhi_{\alpha\circ \beta}$ naturally induces a $C^\ast$ homomorphism 
	\[ \sigma_\ast \colon C_{L,\gamma}^{*}(P_sG,P_{s}\Gamma)^{G, \Gamma} \to C_{P_s, h}. \]
	
	Note that at the $K$-theory level, we have the natural isomorphism 
	\[ K_i(C_{P_s, h}) \cong K_i(C^{*}_{L}(P_{s}G,P_{s}\Gamma)^{G,\Gamma}) \]
	In particular, the map $\sigma$ induces a homomorphism  
	$$\sigma_* \colon  K_i(C_{L,\gamma}^{*}(P_sG,P_{s}\Gamma)^{G, \Gamma}) \rightarrow  K_i(C^{*}_{L}(P_{s}G,P_{s}\Gamma)^{G,\Gamma} ),$$
	which, after passing to the inductive limit, gives 
	a homomorphism
	$$\sigma_* \colon  K_{i,\gamma}^{G,\Gamma}(\underline{E}G,\underline{E}\Gamma) \longrightarrow  K_{i}^{G,\Gamma}(\underline{E}G,\underline{E}\Gamma).$$

	\begin{prop}\label{samehomology}
		If   $G$  has  a (rational) $\gamma$-element, then the map
		$$\sigma_*\colon  K_{i,\gamma}^{G,\Gamma}(\underline{E}G,\underline{E}\Gamma) \longrightarrow  K_{i}^{G,\Gamma}(\underline{E}G,\underline{E}\Gamma)$$
		is a (rational) isomorphism.
	\end{prop}
	\begin{proof}
		The above discussion together with  Lemma \ref{lm:exact} implies  that we have the following commutative diagram: 
		\begin{equation*}
			\begin{tikzcd}
				K_{i+1}(C^*_L(P_sG,\Sigma  ^2)^G )\arrow[r,"(\varPhi_{\alpha\circ \beta})_*"] \arrow[d,"(h^L_{\gamma})_*"] & K_{i+1}(C_{\maximal, L}^*(P_{s}G,\mathcal{K}(H))^G ) \arrow[d,"(h_{\maximal, L})_*"] \\
				K_{*+1}(C_{L}^*(P_{s}\Gamma,\mathcal{K}(H))^\Gamma)   \arrow[r,"\cong"]\arrow[d,"\imath_*"] &K_{*+1}(C_{L}^*(P_{s}\Gamma,\mathcal{K}(H))^\Gamma) \arrow[d,"\imath_*"]\\
				K_*(C_{L,\gamma}^{*}(P_sG,P_{s}\Gamma)^{G, \Gamma}) \arrow[r, "\sigma_*"]\arrow[d,"r_*"] & K_*(C^{*}_{L}(P_{s}G,P_{s}\Gamma)^{G,\Gamma})\arrow[d,"r_*"]\\
				K_{*}(C^*_L(P_sG,\Sigma  ^2)^G )\arrow[r,"(\varPhi_{\alpha\circ \beta})_*"]\arrow[d,"(h^L_{\gamma})_*"] & K_{*}(C_{L}^*(P_sG)^G)  \arrow[d," (h_{\maximal, L})_*"]\\
				K_{*}( C_{L}^*(P_{s}\Gamma,\mathcal{K}(H))^\Gamma)  \arrow[r,"\cong"] &K_{*}(C_{L}^*(P_{s}\Gamma,\mathcal{K}(H))^\Gamma)
			\end{tikzcd}.
		\end{equation*}
		where both columns are exact sequences. By the properties  of the (rational) $\gamma$ element (Definition \ref{def:gammaelement} \& \ref{def:rationalgamma}), it follows by construction that $(\varPhi_{\alpha\circ \beta})_\ast$ is 
		a (rational) isomorphism. By the five lemma, we see that  the map $\sigma_*$ is an (rational) isomorphism. This finishes the proof. 
	\end{proof}
	
	The above discussion leads to the following variant of the relative Baum-Connes assembly map.  
	
	\begin{defn}
		Let  $G$ and $\Gamma$ be fintely generated countable discrete groups. Suppose $G$ has a (rational) $\gamma$ element. We define the \emph{$\gamma$-reduced  relative Baum-Connes assembly map} to be  the homomorphism $$\mu_{\gamma}\colon K_{*}^{G,\Gamma}(\underline{E}G,\underline{E}\Gamma) \cong  K_{*,\gamma}^{G,\Gamma}(\underline{E}G,\underline{E}\Gamma)  \rightarrow K_*(C_{\gamma}^*(G,\Gamma))$$ induced by the evalutation map 
		$$ {\ev}\colon C^*_{L,\gamma}(P_sG, P_s\Gamma)^{G, \Gamma}  \to C_{r,\gamma}^{*}(P_sG,P_{s}\Gamma)^{G, \Gamma}\cong C_{\gamma}^{*}(G, \Gamma)\otimes \mathcal{K}(H).$$
	\end{defn}
	
	Let us state the corresponding strong relative Novikov conjecture as follows. 
	
	\begin{conj}
		Let $h\colon G \to \Gamma$ be a group homomorphism between finitely generated groups. Suppose that $G$ has a $\gamma$-element. Then the $\gamma$-reduced relative Baum-Connes assembly map
		$$\mu_\gamma\colon K_{*}^{G,\Gamma}(\underline{E}G,\underline{E}\Gamma)  \cong K_{*,\gamma}^{G,\Gamma}(\underline{E}G,\underline{E}\Gamma) \to  K_*(C_{\gamma}^*(G,\Gamma))$$
		is injective.
	\end{conj}
	
	Similarly, we have the following rational version. 
	
	\begin{conj}
		Let $h\colon G \to \Gamma$ be a group homomorphism between finitely generated groups. Suppose that $G$ has a rational  $\gamma$-element. Then the $\gamma$-reduced relative Baum-Connes assembly map is rationally injective, that is, 
		$$
		\mu_\gamma \colon K_{*}^{G,\Gamma}(\underline{E}G,\underline{E}\Gamma)\otimes \mathbb{Q}  \cong K_{*,\gamma}^{G,\Gamma}(\underline{E}G,\underline{E}\Gamma) \otimes \mathbb{Q} \to  K_*(C_{\gamma}^*(G,\Gamma))\otimes \mathbb{Q} 
		$$
		is injective.
	\end{conj}
	
	In Section \ref{sec:main}, we shall prove these two conjectures for $h\colon G\to \Gamma$ under the assumption that $G$ has a (rational) $\gamma$-element and $\Gamma$ satisfies the (rational) strong Novikov conjecture. 
	
	\section{The relative Bott map}\label{sec:relativeBott}
	
	In this section, we shall construct relative Bott map from relative $\gamma$-reduced group $C^*$-algebras to relative reduced crossed products.

	Assume that $G$ is a finitely generated group with a (rational) $\gamma$-element. Let us retain the same notation from  Definition \ref{def:gammaelement} and \ref{def:rationalgamma}. We have the following analogue of Lemma \ref{key-lemma}. 
	
	\begin{lem}
		
		Suppose $G$ has a (rational) $\gamma$-element.
		Given  a group homomorphism $h\colon G\rightarrow\Gamma$, there exists a natural asymptotic morphism 
		\begin{equation} \label{eq:homoalphaA}
			h_{\alpha}^\mathcal{A}\colon  C_{\red}^*(G, \Sigma  \mathcal{A}\otimes \mathcal{K}(H))\dashrightarrow C_{\red}^*(\Gamma, \mathcal{K}(H)) 
		\end{equation}
		such that $$h^{\mathcal A}_{\alpha,t}=h^{\mathcal A}_{t} \textup{ on } C_c(G, \Sigma  \mathcal{A}\otimes \mathcal{K}(H)),$$
		where for each $t\in[1,+\infty)$, $h^{\mathcal A}_{t}$ is a map  $C_c(G, \Sigma  \mathcal{A}\otimes \mathcal{K}(H)) \longrightarrow C_c(\Gamma, \mathcal{K}(H))$ given by
		$$h^{\mathcal A}_{t}(\sum a_g\cdot g)=\sum (\alpha_{t}(a_g) \cdot U_g )\cdot h(g).$$
		Here $\alpha$ is the asysmptotic morphism from line \eqref{eq:mapalpha} and $U_g$ stands for the action of $G$ on $H = \oplus_{n=1}^\infty \ell^2(G)$. 
	\end{lem}
	\begin{proof}
		The proof is similar to that of  Lemma \ref{key-lemma}.  The group homomorphism 
		$\widetilde h\colon  G\rightarrow G\times\Gamma$ induces a   $*$-homomorphism (also denoted by $\widetilde h$) $$\widetilde h\colon C_{\red}^*(G,\Sigma  \mathcal{A}\otimes \mathcal{K}(H))\longrightarrow C_{\red}^*(G\times\Gamma, \Sigma  \mathcal{A}\otimes \mathcal{K}(H))$$
		where $G\times\Gamma$ acts on   $\Sigma  \mathcal{A}\otimes \mathcal{K}(H)$ by  $$(g,g^\prime)\cdot a:=g\cdot a,$$ for any $(g,g^\prime)\in G\times\Gamma$ and $a\in \Sigma  \mathcal{A}\otimes \mathcal{K}(H)$. Note that we have the  natural isomorphism
		$$\zeta\colon  C_{\red}^*(G\times\Gamma, \Sigma  \mathcal{A}\otimes \mathcal{K}(H))
		\rightarrow C_{\red}^*(G, \Sigma  \mathcal{A}\otimes \mathcal{K}(H))\otimes C_{\red}^*(\Gamma).$$
		
		Since $\Sigma  \mathcal{A}\otimes \mathcal{K}(H)$ is a proper $G$-$C^*$-algebra,  the maximal crossed product of $\Sigma  \mathcal{A}\otimes \mathcal{K}(H)$ by $G$ coincides with the corresponding reduced crossed product:
		$$\theta \colon  C_{\maximal}^*(G, \Sigma  \mathcal{A}\otimes \mathcal{K}(H))
		\stackrel{\cong}{\longrightarrow}C_{\red}^*(G, \Sigma  \mathcal{A}\otimes \mathcal{K}(H)).$$
		The asymptotic morphism $\alpha$ from line \eqref{eq:mapalpha} induces an asymptotic  morphism 
		$$
		\alpha: C_{\maximal}^*(G, \Sigma  \mathcal{A}\otimes \mathcal{K}(H))
		\dasharrow C_{\maximal}^*(G,\mathcal{K}(H)).
		$$

		We denote by
		$$\pi\colon C_{\maximal}^*(G,\mathcal{K}(H))\rightarrow \mathcal{K}(H)$$ the $*$-homomorphism  induced by  the trivial group homomorphism
		$\pi\colon G\rightarrow \{e\}$.
		
		Now let $h_{\alpha}^\mathcal{A}\colon  C_{\red}^*(G, \Sigma  \mathcal{A}\otimes \mathcal{K}(H))\dashrightarrow C_{\red}^*(\Gamma, \mathcal{K}(H))$ be the asymptotic morphism obtained as the following  composition
		\[\begin{tikzcd}
			C_{\red}^*(G, \Sigma  \mathcal{A}\otimes \mathcal{K}(H)) & C_{\red}^*(G\times\Gamma, \Sigma  \mathcal{A}\otimes \mathcal{K}(H))  \\ 
			& C_{\maximal}^*(G, \Sigma  \mathcal{A}\otimes \mathcal{K}(H)) \otimes C_r^\ast(\Gamma)
			\\
			\mathcal{K}(H) \otimes C_r^\ast(\Gamma) & C_{\maximal}^*(G,\mathcal{K}(H))\otimes C_r^\ast(\Gamma) & 
			\arrow["\widetilde h", from=1-1, to=1-2]
			%\arrow["\beta", dashed, from=1-2, to=1-3]
			%\arrow["=", from=1-3, to=3-3]
			\arrow["\pi\otimes \id "', from=3-2, to=3-1]
			\arrow["\cong", from=1-2, to=2-2]
			\arrow["\alpha\otimes \id ", dashed, from=2-2, to=3-2]
		\end{tikzcd}\]
		It follows by construction that $$h^{\mathcal A}_{\alpha,t}=h^{\mathcal A}_{t} \textup{ on } C_c(G, \Sigma  \mathcal{A}\otimes \mathcal{K}(H)).$$
	\end{proof}

	The asymptotic morphism 
	$$h_{\alpha}^\mathcal{A}\colon  C_{\red}^*(G, \Sigma  \mathcal{A}\otimes \mathcal{K}(H))\dashrightarrow C_{\red}^*(\Gamma, \mathcal{K}(H)) $$ 
	induces a $C^\ast$-homomorphism (still denoted by $h_{\alpha}^\mathcal{A}$) 
	$$h_{\alpha}^\mathcal{A}\colon  C_{\red}^*(G, \Sigma  \mathcal{A}\otimes \mathcal{K}(H)) \longrightarrow \asymalg(C_{\red}^*(\Gamma, \mathcal{K}(H))).$$
	\begin{defn}
		We denote by $C_{\alpha}^*(G,\Gamma, \mathcal{A})$ 
		the mapping cone of $$h_{\alpha}^\mathcal{A}\colon  C_{\red}^*(G, \Sigma  \mathcal{A}\otimes \mathcal{K}(H)) \longrightarrow \asymalg(C_{\red}^*(\Gamma, \mathcal{K}(H))).$$ 
	\end{defn}

	Similar to the discussion of Section \ref{sec:relativeBC}, we have the analogue of the asymptotic morphism 
	\begin{equation*}
		h_{\alpha}^\mathcal{A}\colon  C_{\red}^*(G, \Sigma  \mathcal{A}\otimes \mathcal{K}(H))\dashrightarrow C_{\red}^*(\Gamma, \mathcal{K}(H))
	\end{equation*}
	$$ $$ 
	in terms of Roe algebras and localization algebras with coefficients in $\mathcal A$: 
	\begin{equation}\label{eq:homoalphaARoe}
		h_{\alpha}^\mathcal{A}\colon C_{\red}^*(P_sG, \Sigma  \mathcal{A}\otimes \mathcal{K}(H))^G \dashrightarrow C_{\red}^*(P_s\Gamma, \mathcal{K}(H))^\Gamma
	\end{equation}
	and 
	\begin{equation}\label{eq:localhomoA}
		h_{\alpha}^{L, \mathcal{A}}\colon C_{L}^*(P_sG, \Sigma  \mathcal{A}\otimes \mathcal{K}(H))^G \dashrightarrow C_{L}^*(P_s\Gamma, \mathcal{K}(H))^\Gamma.
	\end{equation}
	Moreover, the following diagram commutes
	\begin{equation*}
		\begin{tikzcd}
			C_{L}^*(P_sG, \Sigma  \mathcal{A}\otimes \mathcal{K}(H))^G\arrow[r, "h_{\alpha,t}^{L, \mathcal{A}}"]\arrow[d,"\ev"]& C_{L}^*(P_s\Gamma, \mathcal{K}(H))^\Gamma\arrow[d,"\ev"] \\
			C_{\red}^*(P_sG, \Sigma \mathcal{A}\otimes \mathcal{K}(H))^G\arrow[r, "h_{\alpha,t}^\mathcal{A}"]& C_{\red}^*(P_s\Gamma, \mathcal{K}(H))^\Gamma\\
		\end{tikzcd}
	\end{equation*}
	for each $t\geq 1$.

		Consider the $C^\ast$-homomorphisms associated to the above asymptotic morphisms: 
		$$h_{\alpha}^\mathcal{A}\colon  C_{\red}^*(P_sG, \Sigma  \mathcal{A}\otimes \mathcal{K}(H))^G \longrightarrow \asymalg(C_{\red}^*(P_s\Gamma, \mathcal{K}(H))^\Gamma)$$
		and
		$$h_{\alpha}^{L,\mathcal{A}}\colon  C_{L}^*(P_sG, \Sigma  \mathcal{A}\otimes \mathcal{K}(H))^G \longrightarrow \asymalg(  C_{L}^*(P_s\Gamma, \mathcal{K}(H))^\Gamma ).$$

		\begin{defn}
			Let 
			$C_{L,\alpha}^{*}(P_sG,P_{s}\Gamma; \mathcal{A})^{G,\Gamma}$ be  the mapping cone of 
			$$ h_{\alpha}^{L,\mathcal{A}}\colon  C_{L}^*(P_sG, \Sigma  \mathcal{A}\otimes \mathcal{K}(H))^G \longrightarrow \asymalg(  C_{L}^*(P_s\Gamma, \mathcal{K}(H))^\Gamma )$$ and $C_{\alpha}^{*}(P_sG,P_{s}\Gamma; \mathcal{A})^{G,\Gamma}$  the mapping cone of  $$h_{\alpha}^\mathcal{A}\colon  C_{\red}^*(P_sG, \Sigma  \mathcal{A}\otimes \mathcal{K}(H))^G \longrightarrow \asymalg(C_{\red}^*(P_s\Gamma, \mathcal{K}(H))^\Gamma)$$
		\end{defn}
		
		\begin{defn}
			We define the \textit{relative equivariant K-homology} with coefficients in $\mathcal{A}$  for  the map $h\colon G\to \Gamma$ to be   $$K_{i,\alpha}^{G,\Gamma}(\underline{E}G,\underline{E}\Gamma; \mathcal{A}):=\lim_{s\rightarrow\infty}K_i(C_{L,\alpha}^{*}(P_sG,P_{s}\Gamma; \mathcal{A})^{G,\Gamma}).$$ Similar to Proposition \ref{prop:Roeiso}, we have 
			$$C_{\alpha}^{*}(P_sG,P_{s}\Gamma;\mathcal{A})^{G,\Gamma}\cong C_{\alpha}^*(G,\Gamma; \mathcal{A})\otimes \mathcal{K}(H).$$
			We define the \textit{relative Baum-Connes assembly map with coefficients in $\mathcal{A}$} to be the homomorphism 
			\begin{equation}\label{eq:assemblycoeff}
				\mu^\mathcal{A}\colon  K_{*,\alpha}^{G,\Gamma}(\underline{E}G,\underline{E}\Gamma;\mathcal{A})\longrightarrow K_*(C_{\alpha}^*(G,\Gamma; \mathcal{A}))
			\end{equation}
			induced by evaluation map 
			$$
			\ev\colon  C_{L,\alpha}^{*}(P_sG,P_{s}\Gamma;\mathcal{A})^{G,\Gamma}\to  C_{\alpha}^{*}(P_sG,P_{s}\Gamma;\mathcal{A})^{G,\Gamma}\cong C_{\alpha}^*(G,\Gamma; \mathcal{A})\otimes \mathcal{K}(H).
			$$
		\end{defn}

		Recall that $C_{{\gamma}}^*(G,\Gamma)$ is the mapping cone algebra associated to the $$h_{\gamma}\colon  \Sigma^2C_{\red}^*(G)\to \asymalg(C_{\red}^*(\Gamma)\otimes  \mathcal{K}(H)) $$ 
		and 
		$C_{L, \gamma}^{*}(P_sG,P_{s}\Gamma)^{G,\Gamma}$ is  the mapping cone of $$
		h^L_{\gamma}:C_{L}^*(P_s(G), \Sigma  ^2)^G \overset{}{\longrightarrow}  \asymalg( C_{L}^*(P_{s}(\Gamma),\mathcal{K}(H))^\Gamma).
		$$
		In the following, we shall introduce the relative Bott maps  
		$$\beta^\mathcal{A}\colon  C_{\gamma}^*(G,\Gamma) \dasharrow C_{\alpha}^*(G,\Gamma; \mathcal{A})$$
		and 
		$$\beta_{L}^\mathcal{A}\colon  C_{L, \gamma}^{*}(P_sG,P_{s}\Gamma)^{G,\Gamma}
		\dasharrow   C_{L,\alpha}^{*}(P_sG,P_{s}\Gamma;\mathcal{A})^{G,\Gamma}.$$
		
		By the above discussion together with that of Section \ref{sec:gamma-groupalgebra},   we have the following commutative  diagram (Figure \ref{fig:diagram1}):
		\begin{figure}[h]
			\begin{tikzcd}
				C_{\red}^*(G, \Sigma^2)  & & C_{\red}^*(G, \Sigma  \mathcal{A}\otimes \mathcal{K}(H))  \\
				& C_{\red}^*(G\times\Gamma, \Sigma^2)  &    \\
				& C_{\red}^*(G\times\Gamma,  \Sigma  \mathcal{A}\otimes \mathcal{K}(H))  \\
				& C_{\red}^*(G, \Sigma  \mathcal{A}\otimes \mathcal{K}(H))\otimes C_{\red}^*(\Gamma) \\
				& C_{\maximal}^*(G, \Sigma  \mathcal{A}\otimes \mathcal{K}(H))\otimes C_{\red}^*(\Gamma) \\
				& C_{\maximal}^*(G,\mathcal{K}(H))\otimes C_{\red}^*\Gamma \\
				& \mathcal{K}(H)\otimes C_{\red}^*(\Gamma)
				\arrow["\cong", from=3-2, to=4-2]
				\arrow["\cong", from=4-2, to=5-2]
				\arrow["{\alpha\otimes \id}", dashed, from=5-2, to=6-2]
				\arrow["{\pi\otimes \id}", from=6-2, to=7-2]
				\arrow["h_{\alpha}^\mathcal{A}", bend left=40, dashed, from=1-3, to=7-2]
				\arrow["\beta", dashed, from=1-1, to=1-3]
				\arrow["{\widetilde h}"', from=1-1, to=2-2]
				\arrow["h_\gamma"', bend right=40, dashed, from=1-1, to=7-2]
				\arrow["\beta", dashed, from=2-2, to=3-2]
				\arrow["{\widetilde h}"', bend left=20, from=1-3, to=3-2]
				%\arrow[from=2-3, to=3-2]
			\end{tikzcd}
			\caption{Asymptotic maps $h_\gamma$ from line \eqref{eq:homogamma} and $h_\alpha^{\mathcal A}$ from line \eqref{eq:homoalphaA}}
			\label{fig:diagram1}
		\end{figure}

		Let us apply the asymptotic morphism $\beta\colon  \Sigma^2  \dasharrow  \Sigma\mathcal{A}\otimes \mathcal{K}(H) $  to construct an asymptotic morphism  
		\[ \beta^{\mathcal A}\colon C_{{\gamma}}^*(G,\Gamma) \dasharrow  C_{\alpha}^*(G,\Gamma; \mathcal{A})\]
		as follows.  For each $(b,f)\in C_{{\gamma}}^*(G,\Gamma)$, where $b\in C_{\red}^*(G, \Sigma^2)$ and $f\in C_0([0,1),\asymalg(C_{\red}^*(\Gamma)\otimes \mathcal{K}(H)))$ such that $h_{\gamma}(b)=f(0)$, we define  
		$$\beta^\mathcal{A}_t(b,f):=(\beta_t(b),\widetilde{f}) \textup{ for each } t\geq 1,$$
		
		where $\widetilde{f}$ is given  by
		$$\widetilde{f}(\theta)=
		\begin{cases}[h_{\alpha,t^\prime+(1-3\theta)(1-t)}^\mathcal{A}(\beta_t(b))]_{t^\prime\geq t}, & \textup{ if }  \theta\in [0,\frac{1}{3}), \\
			\\
			[h_{\alpha,t^\prime}^\mathcal{A}(\beta_{_{(3\theta-1)t^\prime+(2-3\theta)t}}(b))]_{t^\prime\geq t}, &  \textup{ if }  \theta\in [\frac{1}{3},\frac{2}{3}), \\
			\\
			f(3\theta-2)|_{_{t^\prime\geq t} }, &  \textup{ if }  \theta\in [\frac{2}{3},1).
		\end{cases}
		$$
		We also  have commutative diagrams (Figure \ref{fig:diagram2} and \ref{fig:diagram3}).  Similarly, we apply the same formula for $\beta^{\mathcal A}$ above to  define  the corresponding asymptotic morphisms 
		$$\beta^\mathcal{A}\colon  C_{r, \gamma}^{*}(P_sG,P_{s}\Gamma)^{G,\Gamma}
		\dasharrow   C_{\alpha}^{*}(P_sG,P_{s}\Gamma; \mathcal{A})^{G,\Gamma}.$$
		and 
		$$\beta_{L}^\mathcal{A}\colon  C_{L, \gamma}^{*}(P_sG,P_{s}\Gamma)^{G,\Gamma}
		\dasharrow   C_{L,\alpha}^{*}(P_sG,P_{s}\Gamma; \mathcal{A})^{G,\Gamma}.$$
		
		\begin{figure}[h]
			\adjustbox{scale=0.9,center}{
				\begin{tikzcd}
					C_{\red}^*(P_sG, \Sigma  ^2)^G & & C_{\red}^*(P_sG, \Sigma  \mathcal{A}\otimes \mathcal{K}(H))^G  \\
					& C_{\red}^*(P_sG\times P_s\Gamma,\Sigma  ^2)^{G\times\Gamma}  &    \\
					& C_{\red}^*(P_{s}G\times P_s\Gamma,\Sigma  \mathcal{A}\otimes \mathcal{K}(H))^{G\times\Gamma} \\
					& C_{\red}^*(P_{s}G,\Sigma  \mathcal{A}\otimes \mathcal{K}(H))^G\otimes  C_{\red}^*(P_{s}\Gamma)^\Gamma \\
					& C_{\maximal}^*(P_{s}G,\Sigma  \mathcal{A}\otimes \mathcal{K}(H))^G\otimes  C_{\red}^*(P_{s}\Gamma)^\Gamma \\
					& C_{\maximal}^*(P_{s}G,\mathcal{K}(H))^G\otimes C_{\red}^*(P_{s}\Gamma)^\Gamma  \\
					& \mathcal{K}(H)\otimes C_{\red}^*(P_s\Gamma)^\Gamma
					\arrow["\cong", from=3-2, to=4-2]
					\arrow["\cong", from=4-2, to=5-2]
					\arrow["{\alpha\otimes \id}", dashed, from=5-2, to=6-2]
					\arrow["{\pi\otimes \id}", from=6-2, to=7-2]
					\arrow["h_{\alpha}^\mathcal{A}",  bend left=40, dashed, from=1-3, to=7-2]
					\arrow["\beta", dashed, from=1-1, to=1-3]
					\arrow["{\widetilde h}"', from=1-1, to=2-2]
					\arrow["h_\gamma"', bend right=40, dashed, from=1-1, to=7-2]
					\arrow["\beta", dashed, from=2-2, to=3-2]
					\arrow["{\widetilde h}"', bend left=20, from=1-3, to=3-2]
					%\arrow[from=2-3, to=3-2]
				\end{tikzcd}
			}
			\caption{Asymptotic maps $h_\gamma$ from line \eqref{eq:homogammaRoe} and $h_\alpha^{\mathcal A}$ from line \eqref{eq:homoalphaARoe}} 
			\label{fig:diagram2}
		\end{figure}
		\begin{figure}[h]
			\adjustbox{scale=0.9,center}{\begin{tikzcd}
					C_{L}^*(P_sG, \Sigma  ^2)^G & & C_{L}^*(P_sG, \Sigma  \mathcal{A}\otimes \mathcal{K}(H))^G  \\
					& C_{L}^*(P_sG\times P_s\Gamma,\Sigma  ^2)^{G\times\Gamma}  &    \\
					& C_{L}^*(P_{s}G\times P_s\Gamma,\Sigma  \mathcal{A}\otimes \mathcal{K}(H))^{G\times\Gamma} \\
					& C_{L}^*(P_{s}G,\Sigma  \mathcal{A}\otimes \mathcal{K}(H))^G\otimes  C_{L}^*(P_{s}\Gamma)^\Gamma \\
					& C_{\maximal, L}^*(P_{s}G,\Sigma  \mathcal{A}\otimes \mathcal{K}(H))^G\otimes  C_{L}^*(P_{s}\Gamma)^\Gamma \\
					& C_{\maximal, L}^*(P_{s}G,\mathcal{K}(H))^G\otimes C_{L}^*(P_{s}\Gamma)^\Gamma  \\
					& C_{L}^*(pt, \mathcal{K}(H))\otimes C_{L}^*(P_s\Gamma)^\Gamma \\
					& \mathcal{K}(H)\otimes C_{L}^*(P_s\Gamma)^\Gamma
					\arrow["\cong", from=3-2, to=4-2]
					\arrow["\cong", from=4-2, to=5-2]
					\arrow["{\alpha\otimes \id}", dashed,  from=5-2, to=6-2]
					\arrow["{\pi\otimes \id}", from=6-2, to=7-2]
					\arrow["{\ev \otimes \id }", from=7-2, to=8-2]
					\arrow["h_{\alpha}^{L, \mathcal{A}}", bend left=40, dashed, from=1-3, to=8-2]
					\arrow["\beta", dashed, from=1-1, to=1-3]
					\arrow["{\widetilde h}"', from=1-1, to=2-2]
					\arrow["h^L_\gamma"', bend right=40, dashed, from=1-1, to=8-2]
					\arrow["\beta", dashed, from=2-2, to=3-2]
					\arrow["{\widetilde h}"', bend left=20, from=1-3, to=3-2]
					%\arrow[from=2-3, to=3-2]
				\end{tikzcd}
			}
			\caption{The Asymptotic maps $h^L_\gamma$ from line \eqref{eq:localgamma} and $h_\alpha^{L, \mathcal A}$ from line \eqref{eq:localhomoA}} 
			\label{fig:diagram3}
		\end{figure}

				By definitions, the following diagram commutes for all  $t\geq 1$ and $s>0$:
				\begin{equation*}
					\begin{tikzcd}
						C_{L,\gamma}^{*}(P_sG,P_{s}\Gamma)^{G, \Gamma} \ar{r}{{\ev}} \ar{d}{\beta_{t}^{L, \mathcal{A}}}
						&  C_{r,\gamma}^{*}(P_sG,P_{s}\Gamma)^{G,\Gamma} \ar{r}{\cong} \ar{d}{\beta_{t}^\mathcal{A}}& C_{\gamma}^*(G,\Gamma) \otimes \mathcal{K}(H)\ar{d}{\beta^\mathcal{A}_t}\\
						C_{L,\alpha}^{*}(P_sG,P_{s}\Gamma,\mathcal{A})^{G,\Gamma} \ar{r}{{\ev}}&
						C_{\alpha}^{*}(P_sG,P_{s}\Gamma,\mathcal{A})^{G,\Gamma} \ar{r}{\cong}& C_{\alpha}^*(G,\Gamma, \mathcal{A})\otimes \mathcal{K}(H).
					\end{tikzcd}
				\end{equation*}

				\begin{defn}\label{relative-Bott-map}
					The relative Bott map $\beta^\mathcal{A}$ above induces a homomorphism at the $K$-theory level: 
					$$(\beta^\mathcal{A})_\ast \colon  K_i(C_{\gamma}^*(G,\Gamma)) \to   K_i(C_{\alpha}^*(G,\Gamma, \mathcal{A}))$$
					Similarly, the relative Bott map $\beta_L^\mathcal{A}$ induces 
					$$
					(\beta_{L}^\mathcal{A})_*\colon  K_i(C_{L,\gamma}^{*}(P_sG,P_{s}\Gamma)^{G, \Gamma}) \to  K_i(C_{L,\alpha}^{*}(P_sG,P_{s}\Gamma,\mathcal{A})^{G,\Gamma} ),
					$$
					which gives 
					$$(\beta_{L}^\mathcal{A})_* \colon  K_{i,\gamma}^{G,\Gamma}(\underline{E}G,\underline{E}\Gamma) 
					\longrightarrow  K_{i,\alpha}^{G,\Gamma}(\underline{E}G,\underline{E}\Gamma,\mathcal{A})$$
					when taking  the inductive limit over $(P_sG, P_s\Gamma)$ for all $s$. 
				\end{defn}

				Observe that the same discussion above  allows us to use  the asymptotic morphism $\alpha \colon  \Sigma  \mathcal{A}\otimes \mathcal{K}(H)\dasharrow  \mathcal{K}(H)$  to construct an asymptotic morphism  
				\[ \alpha_L\colon C_{L,\alpha}^{*}(P_sG,P_{s}\Gamma,\mathcal{A})^{G,\Gamma}  \dasharrow  C_{L}^{*}(P_sG,P_{s}\Gamma)^{G,\Gamma}. \]
				See Figure \ref{fig:diagram4}. 
				
				\begin{figure}[h]
					\adjustbox{scale=0.9,center}{
						\begin{tikzcd}
							C_{L}^*(P_sG, \Sigma  \mathcal{A}\otimes \mathcal{K}(H))^G & C_{L, \maximal}^*(P_sG, \Sigma  \mathcal{A}\otimes \mathcal{K}(H))^G & C_{L, \maximal }^*(P_{s}G,\mathcal{K}(H))^G  \\
							& C_{L}^*(P_{s}G\times P_s\Gamma,\Sigma  \mathcal{A}\otimes \mathcal{K}(H))^{G\times\Gamma} &    \\
							& C_{L}^*(P_{s}G,\Sigma  \mathcal{A}\otimes \mathcal{K}(H))^G\otimes  C_{L}^*(P_{s}\Gamma)^\Gamma \\
							& C_{L, \maximal}^*(P_{s}G,\Sigma  \mathcal{A}\otimes \mathcal{K}(H))^G\otimes  C_{L}^*(P_{s}\Gamma)^\Gamma \\
							& C_{L,\maximal}^*(P_{s}G,\mathcal{K}(H))^G\otimes C_{L}^*(P_{s}\Gamma)^\Gamma \\
							& C_{L}^\ast(pt, \mathcal{K}(H))\otimes C_{L}^*(P_s\Gamma)^\Gamma \\
							& \mathcal{K}(H)\otimes C_{L}^*(P_s\Gamma)^\Gamma
							\arrow["\cong", from=3-2, to=4-2]
							\arrow["\alpha\otimes \id", dashed, from=4-2, to=5-2]
							\arrow["{\pi\otimes \id}", from=5-2, to=6-2]
							\arrow["{\ev}", from=6-2, to=7-2]
							\arrow["h_{\maximal, L}",  bend left=40,  from=1-3, to=7-2]
							\arrow["\cong",  from=1-1, to=1-2]
							\arrow["\alpha", dashed, from=1-2, to=1-3]
							\arrow["{\widetilde h}"', from=1-1, to=2-2]
							\arrow["h_{\alpha}^{L,\mathcal{A}}"', bend right=40, dashed, from=1-1, to=7-2]
							\arrow["\cong",  from=2-2, to=3-2]
							%\arrow["{\widetilde h}"', bend left=20, from=1-3, to=3-2]
							%\arrow[from=2-3, to=3-2]
						\end{tikzcd} 
					}
					\caption{The asymptotic map $h^{L, \mathcal A}_\alpha$ from line \eqref{eq:localhomoA} and the $C^*$ homomorphism $h_{\maximal, L}$ from line \eqref{eq:localhomo}} 
					\label{fig:diagram4}
				\end{figure}
				
				The following proposition is one of the important ingredients for the proof of our main theorem. 
				
				\begin{prop}\label{prop:bottinjective}
					If $G$ has a (rational) $\gamma$-element, then the relative Bott map
					$$
					(\beta_{L}^\mathcal{A})_\ast\colon  K_{i,\gamma}^{G,\Gamma}(\underline{E}G,\underline{E}\Gamma) \to K_{i,\alpha}^{G,\Gamma}(\underline{E}G,\underline{E}\Gamma,\mathcal{A})
					$$
					is (rationally) injective.
				\end{prop}
				\begin{proof} It suffices to prove that 
					$$
					(\beta_{L}^\mathcal{A})_*\colon  K_i(C_{L,\gamma}^{*}(P_sG,P_{s}\Gamma)^{G, \Gamma})\rightarrow K_i(C_{L,\alpha}^{*}(P_sG,P_{s}\Gamma,\mathcal{A})^{G,\Gamma} ).
					$$
					is a (rational) injection for each $s>0$.
					
					Recall that for  a given uniformly continuous  asymptotic morphism  $\varphi\colon A\dasharrow B$ between  separable $C^*$-algebras, we denote  the induced homomorphism $A\to  \asymalg(B)$ still by $\varphi$. We denote by 
					$C_\varphi$ the mapping cone associated to $\varphi \colon A\to  \asymalg(B)$, that is, 
					\[ C_\varphi = \{(a, f) \in A\oplus C_0([0, 1), \asymalg(B)) \mid \varphi (a) = f(0)\}. \]
					In this case, we have the natural injection  $\imath\colon  \Sigma  B\hookrightarrow C_\varphi$   and the natural restriction map  $r\colon C_\varphi \rightarrow A$, cf. Lemma \ref{lm:exact}.
					
					By the definition of the asymptotic morphism $$\beta_{L}^\mathcal{A}\colon  C_{L, \gamma}^{*}(P_sG,P_{s}\Gamma)^{G,\Gamma}
					\dasharrow   C_{L,\alpha}^{*}(P_sG,P_{s}\Gamma,\mathcal{A})^{G,\Gamma}, $$
					we have the following commutative diagram: 
					
					\[\begin{tikzcd}
						K_{i+1}(C_{L}^*(P_sG, \Sigma  ^2)^G)  & K_{i+1}(C_{L}^*(P_sG, \Sigma  \mathcal{A}\otimes \mathcal{K}(H))^G) & K_{i+1}(C_{\maximal, L}^*( P_sG)^G) \\
						K_{i+1}(C_{L}^*(P_{s}\Gamma)^\Gamma) & K_{i+1}(C_{L}^*(P_{s}\Gamma)^\Gamma) & K_{i+1}(C_{L}^*(P_{s}\Gamma)^\Gamma)  \\
						K_i(C_{L,\gamma}^{*}(P_sG,P_{s}\Gamma)^{G, \Gamma}) & K_*(C_{L,\alpha}^{*}(P_sG,P_{s}\Gamma,\mathcal{A})^{G,\Gamma}  ) & K_*(C_{L}^{*}(P_sG,P_{s}\Gamma)^{G,\Gamma}  ) \\
						K_{i}(C_{L}^*(P_sG, \Sigma  ^2)^G ) & K_{i}(C_{L}^*(P_sG, \Sigma  \mathcal{A}\otimes \mathcal{K}(H))^G) & K_{i}(C_{\maximal, L}^*( P_sG)^G)  \\
						K_{i}(C_{L}^*(P_{s}\Gamma)^\Gamma) & 	K_{i}(C_{L}^*(P_{s}\Gamma)^\Gamma) & 	K_{i}(C_{L}^*(P_{s}\Gamma)^\Gamma) 
						\arrow["\beta_*", from=1-1, to=1-2]
						\arrow["\alpha_*", from=1-2, to=1-3]
						\arrow["(h^L_\gamma)_*", from=1-1, to=2-1]
						\arrow["(h^{L, \mathcal A}_\alpha)_*", from=1-2, to=2-2]
						\arrow["(h_{\maximal, L})_*", from=1-3, to=2-3]
						\arrow["=", from=2-1, to=2-2]
						\arrow["=", from=2-2, to=2-3]
						\arrow["\iota_\ast", from=2-1, to=3-1]
						\arrow["\iota_\ast", from=2-2, to=3-2]
						\arrow["\iota_\ast", from=2-3, to=3-3]
						\arrow["(\beta_L^{\mathcal A})_*", from=3-1, to=3-2]
						\arrow["(\alpha_L)_\ast", from=3-2, to=3-3]
						\arrow["r_*", from=3-1, to=4-1]
						\arrow["r_*",from=3-2, to=4-2]
						\arrow["r_*", from=3-3, to=4-3]
						\arrow["\beta_*",from=4-1, to=4-2]
						\arrow["\alpha_*",from=4-2, to=4-3]
						\arrow["(h^L_\gamma)_*", from=4-1, to=5-1]
						\arrow["(h^{L, \mathcal A}_\alpha)_*",from=4-2, to=5-2]
						\arrow["(h_{\maximal, L})_*", from=4-3, to=5-3]
						\arrow["=", from=5-1, to=5-2]
						\arrow["=",from=5-2, to=5-3]
					\end{tikzcd}\]
					where each $\iota_\ast$ is induced by  $\iota$ in the correpsonding mapping cone and similarly  $r_\ast$ is induced by $r$ in the corresponding mapping cone.

					Since by assumption the compositions $\alpha_*\circ\beta_*$ in the first and forth rows  are (rational) isomorphisms, it follows from the five lemma that the composition  $(\alpha_{L})_*\circ (\beta_{L}^\mathcal{A})_*$ from the third row  is also a (rational) isomorphism. Therefore,  $(\beta_{L}^\mathcal{A})_*$ is (rationally) injective. This finishes the proof. 
				\end{proof}

				\section{Main results}\label{sec:main}
				
				In this section, we prove the main results of the paper. 
				
				\begin{thm}\label{main-technical-result}
					Let  $h\colon G\rightarrow \Gamma$ be a group homomorphism between finitely generated groups. 
					If $G$ has a $\gamma$-element (Definition \ref{def:gammaelement}) and $\Gamma$ satisfies the strong Novikov conjecture, then the strong relative Novikov conjecture holds for $h\colon G\rightarrow \Gamma$, i.e.,  the reduced relative Baum-Connes assembly map
					$$
					\mu_\gamma \colon  K_{i}^{G,\Gamma}(\underline{E}G,\underline{E}\Gamma)\rightarrow K_i(C_{\gamma}^*(G,\Gamma))
					$$
					is  injective for $i=0, 1$.
				\end{thm}

				\begin{thm}\label{main-technical-result-rational}
					Let  $h\colon G\rightarrow \Gamma$ be a group homomorphism between finitely generated groups. 
					If $G$ has a rational $\gamma$-element (Definition \ref{def:rationalgamma}) and $\Gamma$ satisfies the rational strong Novikov conjecture,
					then the rational strong relative Novikov conjecture holds, i.e. the reduced relative Baum-Connes assembly map
					$$
					\mu_\gamma \colon  K_{i}^{G,\Gamma}(\underline{E}G,\underline{E}\Gamma)\otimes \mathbb{Q} \rightarrow K_i(C_{\gamma}^*(G,\Gamma))\otimes \mathbb{Q} 
					$$
					is injective for $i=0, 1$. 
				\end{thm}
				
				Tu showed that if  a group $G$ is coarsely embeddable into Hilbert space, then $G$ has a $\gamma$ element \cite{Tu_gamma}. Therefore,  we have the following immediate consequence of Theorem \ref{main-technical-result} and Theorem  \ref{main-technical-result-rational}. 
				
				\begin{thm}\label{main-theorem-body}
					Let  $h\colon G\rightarrow \Gamma$ be a group homomorphism between two countable discrete groups.  If both $G$ and $\Gamma$ are coarsely embeddable  into Hilbert space,
					then the  strong relative Novikov conjecture holds for $(G, \Gamma, h)$, that is, the relative Baum--Connes assembly map
					$$
					\mu_\gamma \colon K_{i}^{G,\Gamma}(\underline{E}G,\underline{E}\Gamma) \rightarrow K_i(C_{\gamma}^{*}(G,\Gamma)) 
					$$
					is  injective for $i=0, 1$. 
				\end{thm}

				Now let us prove Theorem \ref{main-technical-result}. 
				\begin{proof}[Proof of Theorem \ref{main-technical-result}]
					Let us retain the same notation from Section \ref{sec:gamma-groupalgebra}, in particular, Definition \ref{def:gammaelement}. Since $G$ has a $\gamma$-element, it follows from the discussion in the previous sections that we have the following commutative diagram:
					
					\[\begin{tikzcd}
						K_{i}^{G,\Gamma}(\underline{E}G,\underline{E}\Gamma) & K_{i,\gamma}^{G,\Gamma}(\underline{E}G,\underline{E}\Gamma) & K_i(C_{\gamma}^*(G,\Gamma)) \\
						& K_{i,\alpha}^{G,\Gamma}(\underline{E}G,\underline{E}\Gamma; \mathcal{A}) & K_{i,\alpha}^{G,\Gamma}(\underline{E}G,\underline{E}\Gamma,\mathcal{A})
						\arrow["\cong", from=1-1, to=1-2]
						\arrow["\mu_\gamma", from=1-2, to=1-3]
						\arrow["(\beta_{L}^\mathcal{A})_*", from=1-2, to=2-2]
						\arrow["(\beta^\mathcal{A})_*", from=1-3, to=2-3]
						\arrow["\mu^{\mathcal A}", from=2-2, to=2-3]
						\arrow["\mu_\gamma", bend left=30, from=1-1, to=1-3]
					\end{tikzcd}\]
					where $\mu^{\mathcal A}$ is the assembly map from line \eqref{eq:assemblycoeff}. It follows from  Proposition \ref{prop:bottinjective} that $(\beta_{L}^\mathcal{A})_*$ is  injective. Now to show that $\mu$ is  injective, it suffices to show that $\mu^\mathcal{A}$ is  injective.

					Consider the following commutative diagram: 
					\begin{equation*}
						\begin{tikzcd}
							K_{i+1}^{G}(\underline{E}G; \Sigma  \mathcal{A}\otimes \mathcal{K}(H))  \arrow[d,"(h_{\alpha}^{L, \mathcal{A}})_\ast "]\arrow[r,"(\mu_G)^{\mathcal{A}}"]&  K_{i+1}(C_{\red}^*(G; \Sigma  \mathcal{A}\otimes \mathcal{K}(H)))  \arrow[d," (h_{\alpha}^\mathcal{A})_*"]\\
							K_{i+1}^{\Gamma}(\underline{E}\Gamma) \arrow[d,"\imath_*"]\arrow[r,"\mu_{_{\Gamma}}"]&  K_{i+1}(C_{\red}^*\Gamma)\arrow[d,"\imath_*"] \\
							K_{i,\alpha}^{G,\Gamma}(\underline{E}G,\underline{E}\Gamma; \mathcal{A})\arrow[d,"r_*"]\arrow[r,"\mu^{\mathcal{A}}"]  &K_i(C_{\alpha}^*(G,\Gamma;  \mathcal{A}))\arrow[d,"r_*"] \\
							K_{i}^{G}(\underline{E}G; \Sigma  \mathcal{A}\otimes \mathcal{K}(H))   \arrow[d,"(h_{\alpha}^{L, \mathcal{A}})_*"]\arrow[r,"(\mu_{G})^{\mathcal{A}}"]&K_{i}(C_{\red}^*(G; \Sigma  \mathcal{A}\otimes \mathcal{K}(H)))\arrow[d,"(h_{\alpha}^\mathcal{A})_*"] \\
							K_{*}^{\Gamma}(\underline{E}\Gamma)\arrow[r,"\mu_{_{\Gamma}}"]&  K_{*}(C_{\red}^*\Gamma)
						\end{tikzcd}
					\end{equation*}
					where $(\mu_G)^{\mathcal A}$ is the Baum-Connes assembly map for $G$ with coefficients in $\Sigma\mathcal A\otimes \mathcal K(H)$,  $\mu_\Gamma$ is the Baum-Connes assembly map for $\Gamma$, each $\iota_\ast$ is induced by  $\iota$ in the correpsonding mapping cone and similarly  $r_\ast$ is induced by $r$ in the corresponding mapping cone (cf. the proof of Proposition \ref{prop:bottinjective}).  Since $\mathcal A$ is a proper $G$-$C^\ast$ algebra, it follows that 
					the map  $(\mu_{G})^{\mathcal{A}}$ is an isomorphism (cf. Theorem \ref{thm:pGHT}).  Since $\Gamma$ is assumed to be coarsely embeddable into Hilbert space, it follows from a theorem of Yu that  $\mu_{_{\Gamma}}$ is injective \cite{Yu_coa_emb}. Now by the five lemma, the map $\mu^\mathcal{A}$ in the third row
					is injective. This finishes the proof. 
				\end{proof}
				
				The proof of Theorem \ref{main-technical-result-rational} is completely similar to that of Theorem \ref{main-technical-result}
				above, and will be omitted. 
				
				\nocite{MR4300553, tian-thesis, MR4170653, Weinberger-Xie-Yu-additivity-of-higher-rho, Xie-Yu-Higher-invariant-in-NCG, Yu_localization, Skan_Tu_Yu, Lesch-Moscovici-Pflaum_Connes-Chern-charactor, Lei-Lott-Piaz, MR3928083, MR4688091}
				\bibliographystyle{plain}
				\bibliography{ref}

			\end{document}